\documentclass[a4paper,12pt]{amsart}

\usepackage{amssymb}
\usepackage{amsmath}

\newtheorem{theorem}{Theorem}[section] 
\newtheorem{lemma}[theorem]{Lemma} 
\newtheorem{corollary}[theorem]{Corollary} 
\newtheorem{proposition}[theorem]{Proposition}
\newtheorem{conjecture}[theorem]{Conjecture} 
\theoremstyle{definition}
\newtheorem{definition}[theorem]{Definition} 
\newtheorem{remark}[theorem]{Remark}

\newtheorem{example}[theorem]{Example}

\title[]{Low degree morphisms of $E(5,10)$-generalized Verma modules}

\author{Nicoletta Cantarini}\author{Fabrizio Caselli}

\address{Fabrizio Caselli and Nicoletta Cantarini, Dipartimento di matematica, Universit\`a di Bologna, Piazza di Porta San Donato 5, 40126 Bologna, Italy}

\email{fabrizio.caselli@unibo.it}
\email{nicoletta.cantarini@unibo.it}

\subjclass[2010]{17B15, 17B25 (primary), 17B65, 17B70 (secondary)}
\keywords{Lie superalgebras, Verma modules, Singular vectors}
 \DeclareMathOperator{\Sym}{Sym}
  \DeclareMathOperator{\Hom}{Hom}
\newcommand{\N}{\mathbb{N}}
\newcommand{\C}{\mathbb{C}}
\newcommand{\Z}{\mathbb{Z}}
\newcommand{\slc}{\mathfrak{sl}_5}
\newcommand{\inlinewedge}{\textrm{\raisebox{0.6mm}{\footnotesize $\bigwedge$}}}
\newcommand{\displaywedge}{\textrm{\raisebox{0.6mm}{\tiny $\bigwedge$}}}

\begin{document}
\maketitle
\begin{abstract} In this paper we face the study of the representations of the exceptional Lie superalgebra $E(5,10)$.
We recall the construction of generalized Verma modules and give a combinatorial description of the restriction to $\mathfrak{sl}_5$ of the Verma module induced
by the trivial representation. We use this description to classify morphisms between Verma modules of degree one, two and three
proving in these cases a conjecture given by Rudakov \cite{R}. A key tool is the notion of dual morphism between Verma modules. 
\end{abstract}
\section{Introduction}
Infinite dimensional linearly compact simple Lie superalgebras over the complex numbers were classified by Victor Kac in 1998 \cite{K}. A complete list, up to isomorphisms, consists of ten
infinite series and five exceptions, denoted by $E(1,6)$, $E(3,6)$, $E(3,8)$, $E(5,10)$ and $E(4,4)$. See also \cite{CK, S1, S2, S3} for the genesis of these superalgebras.
Some years later Kac and Rudakov
initiated the study of the representations of these algebras \cite{KR1, KR2, KR3, KR} developing a general theory of Verma modules that we briefly recall.

Let $L=\oplus_{j\in\Z}L_j$ be a $\Z$-graded Lie superalgebra, let $L_-=\oplus_{j<0}L_j$, $L_+=\oplus_{j>0}L_j$ and $L_{\geq 0}=L_0\oplus L_+$.
We denote by $U(L)$ the universal enveloping algebra of $L$.
If $F$ is an irreducible $L_0$-module  we define
$$M(F)=U(L)\otimes_{U(L_{\geq 0})} F$$
where we extend the action of $L_0$ to $L_{\geq 0}$ by letting $L_+$ act trivially on $F$.
We call $M(F)$ a  minimal generalized Verma module associated to $F$. If $M(F)$ is not irreducible we say that it is degenerate.

In \cite{KR1, KR2, KR3, KR}, a complete description of the degenerate Verma modules for $E(3,6)$ and
$E(3,8)$ is given, as well as of their unique irreducible quotients. 
In \cite{KR} some basic ideas and constructions are settled also for $E(5,10)$. In this case Kac and Rudakov conjecture a complete list of $L_0$-modules which
give rise to the degenerate
Verma modules (see Conjecture \ref{conjecture}).

In 2010 Rudakov tackled the proof of the conjecture through the study of morphisms between Verma modules. The existence of a degenerate Verma module
is indeed strictly related to the existence of such morphisms of positive degree (see Proposition \ref{morphism}). In \cite{R} Rudakov classified 
morphisms of degree one  and gave some examples of morphisms of degree at most 5. He also conjectured that there exists no morphism of higher degree and
that his list exhausts all the examples. A more general family of modules,
possibly induced from infinite-dimensional $\slc$-modules,
had been studied in \cite{GLS}, where some of Rudakov's examples in degree one and two had been obtained through the use of the
computer. 

In this paper we study morphisms between generalized Verma modules and to this aim we analyze the structure of the universal enveloping algebra $U_-=
U(L_-)$
as an $L_0$-module.
This analysis has its own interest and
provides an explicit combinatorial description of the action of $L_0$. This description is the main ingredient in our study of morphisms,
together with a systematic use of the dominance order of the weights of the $L_0$-modules.
Our main result is the proof of Rudakov's conjecture in degree two and three (see Theorems \ref{teorema2}, \ref{teorema3}).
A useful observation that we made is that if there exists a morphism $\varphi: M(V) \rightarrow M(W)$
between generalized Verma modules of degree $d$, then there exists a dual morphism $\psi: M(W^*) \rightarrow M(V^*)$ of the same degree. 
This duality is here proved in low degree for the purpose of this work but it holds in a much wider context as a consequence of the fact that
the conformal dual of a Verma module is itself a Verma module. This will be shown in  a forthcoming paper.

The paper is organized as follows: in Section \ref{S1} we recall the basic definitions and fix the notation. Section \ref{S3} is dedicated to Verma modules. Here we
characterize degenerate Verma modules in terms of singular vectors and morphisms. In Section \ref{S4}, following \cite{R}, we give examples of morphisms of degree one, two and three.
Section \ref{S5} contains our first main result on the structure of $U_-$ as an $L_0$-module: we construct an explicit basis of $U_-$ and describe
its combinatorial properties. Section \ref{S6} is dedicated to the analysis of the dominance order of the weights of the basis elements of $U_-$.
In Section \ref{six} we develop the idea of dual morphism between generalized Verma modules and establish sufficient conditions for the existence of such a morphism 
(see Remark \ref{lemdual2}). Finally, Sections \ref{S7}, \ref{eight} and \ref{nine} contain the classification of morphisms of degree one, two and three, respectively.
 
We thank Victor Kac for useful discussions.

\section{Preliminaries}\label{S1}
We let $\N=\{0,1,2,3,\dots\}$ be the set of non-negative integers and for $n\in\N$ we set $[n]=\{i\in\N ~|~ 1\leq i\leq n\}$.

If $P$ is a proposition we let $\chi_{P}=1$ if $P$ is true and $\chi_{P}=0$ if $P$ is false.

We consider the simple, linearly compact Lie superalgebra of exceptional type $L=E(5,10)$ whose even and odd parts are
as follows: $L_{\bar{0}}$ consists of zero-divergence vector fields in five (even) indeterminates $x_1,\ldots,x_5$, i.e., 
\[L_{\bar{0}}=S_5=\{X=\sum_{i=1}^5f_i\partial_i ~|~ f_i\in\C[[x_1,\dots,x_5]], \textrm{div}(X)=0\},\]
where $\partial_i=\partial_{x_i}$,
and $L_{\bar{1}}=\Omega^2_{cl}$ consists of closed two-forms in the five indeterminates $x_1,\ldots,x_5$.
The bracket between a vector field and a form is given by the Lie derivative and for $f,g\in \C[[x_1,\dots,x_5]]$ we have
$$[fdx_i\wedge dx_j,g dx_k\wedge dx_l]=\varepsilon_{ijkl}fg\partial_{t_{ijkl}}$$ 
where, for $i,j,k,l\in [5]$, $\varepsilon_{ijkl}$ and $t_{ijkl}$ are defined as follows: if $|\{i,j,k,l\}|=4$ we let $t_{ijkl}\in [5]$ be such that $|\{i,j,k,l,t_{ijkl}\}|=5$ and $\varepsilon_{ijkl}$ be the sign of the permutation 
$(i,j,k,l,t_{ijkl})$. If $|\{i,j,k,l\}|<4$ we let $t_{ijkl}=1$ (this choice will be irrelevant) and $\varepsilon_{ijkl}=0$. 

From now on we shall denote $dx_i\wedge dx_j$ simply by $d_{ij}$.

The Lie superalgebra $L$ has a consistent, irreducible, transitive $\Z$-grading of depth 2 where,
for $k\in\N$,
\begin{align*}
L_{2k-2}&=\langle f\partial_i ~|~i=1,\dots,5, f\in\C[[x_1,\dots, x_5]]_{k}\rangle\cap S_5\\
L_{2k-1}&=\langle fd_{ij} ~|~ i,j=1,\dots,5, f\in\C[[x_1,\dots, x_5]]_{k}\rangle\cap\Omega^2_{cl}
\end{align*}
where 
by $\C[[x_1,\dots, x_5]]_{k}$ we denote the homogeneous component of $\C[[x_1,\dots, x_5]]$ of degree $k$.

Note that $L_0\cong \mathfrak{sl}_5$, $L_{-2}\cong (\C^5)^*$, $L_{-1}\cong \inlinewedge^2\C^5$ as $L_0$-modules (where $\C^5$ denotes the
standard $\mathfrak{sl}_5$-module).
We set $L_{-}=L_{-2}\oplus L_{-1}$, $L_{+}=\oplus_{j>0}L_j$ and $L_{\geq 0}=L_0\oplus L_+$.
We denote by $U$ (resp.\ $U_{-}$) the universal enveloping algebra of $L$ (resp.\ $L_-$). Note that $U_-$ is
an $L_0$-module with respect to the adjoint action: for $x\in L_0$ and $u\in U_-$,
$$x.u=[x,u]=xu-ux.$$
We also point out that the $\Z$-grading of $L$ induces a $\Z$-grading on the enveloping algebra $U_-$.
It is customary, though, to invert the sign of the degrees hence getting a grading over $\N$. Note that the homogeneous component
$(U_-)_d$ of degree $d$
of $U_-$ under this grading is an $L_0$-submodule. Section \ref{S3} will be dedicated to the study of
these homogeneous components.

We fix the Borel subalgebra $\langle x_i\partial_j, h_{ij}=x_i\partial_i-x_j\partial_j ~|~ i<j\rangle$ of $L_0$ and we consider the usual base of the corresponding root system given by $\{\alpha_{12},\ldots,\alpha_{45}\}$. We let $\Lambda$ be the weight lattice of $\frak{sl}_5$ and we express all weights of $\frak{sl}_5$ using their coordinates with respect to the fundamental weights $\omega_{12},\omega_{23},\omega_{34},\omega_{45}$, i.e., for $\lambda\in \Lambda$ we write $\lambda=(\lambda_{12},\ldots,\lambda_{45})$ for some $\lambda_{i\,i+1}\in \mathbb Z$ to mean $\lambda=\lambda_{12}\omega_{12}+\cdots+\lambda_{45}\omega_{45}$.

For $i<j$ we denote as usual
\[
\alpha_{ij}=\sum_{k=i}^{j-1}\alpha_{k\,k+1}
\]
and $\alpha_{ji}=-\alpha_{ij}$. For notational convenience we also let $\alpha_{ii}=0$.
Viewed as elements in the weight lattice we have 
\[
\alpha_{12}=(2,-1,0,0),\,\alpha_{23}=(-1,2,-1,0),\,\alpha_{34}=(0,-1,2,-1),\, \alpha_{45}=(0,0,-1,2).
\]
 If $\lambda\in \Lambda$ is a weight,  we use the following convention: for all $1\leq i<j\leq 5$ we let
\[
\lambda_{ij}=\sum_{k=i}^{j-1}\lambda_{k\,k+1}.
\]
If $V$ is a $\frak {sl}_5$-module and $v\in V$ is a weight vector we denote by $\lambda(v)$ the weight of $v$ and by $\lambda_{ij}(v)=(\lambda(v))_{ij}$.

If $\lambda=(a,b,c,d)\in \Lambda$ is a dominant weight, i.e. $a,b,c,d\geq 0$, let us denote by $F(\lambda)=F(a,b,c,d)$ the irreducible $\mathfrak{sl}_5$-module of highest 
weight $\lambda$. In this paper we always think of $F(a,b,c,d)$ as the irreducible submodule of \[\Sym^a(\C^5)\otimes \Sym^b(\displaywedge^2(\C^5))\otimes 
\Sym^c(\displaywedge^2(\C^5)^*)\otimes \Sym^d((\C^5)^*)\] generated by the highest weight vector $x_1^ax_{12}^b{x_{45}^*}^c{x_5^*}^d$
where  $\{x_1,\dots, x_5\}$ denotes the standard basis of $\C^5$, $x_{ij}=x_i\wedge x_j$, and $x_i^*$ and $x_{ij}^*$ are the corresponding dual
basis elements. Besides, for a weight $\lambda=(a,b,c,d)$ we let  $\lambda^*=(d,c,b,a)$, so that $F(\lambda)^*\cong F(\lambda^*)$.
 
Notice that $L_1\cong F(1,1,0,0)$ and that $x_5d_{45}$ is a lowest weight vector in $L_1$. Moreover,
for $j\geq 1$, we have $L_j=L_1^j$.

\section{Generalized Verma modules and morphisms}\label{S3}
We recall the definition of generalized Verma modules introduced in \cite{KR1}. For the reader's convenience we also sketch some proofs of basic results.
Given an $L_0$-module $V$ we extend it to
an $L_{\geq 0}$-module by letting $L_+$ act trivially, and define
$$M(V)=U\otimes_{U(L_{\geq 0})}V.$$
Note that $M(V)$ has a $L$-module structure by multiplication on the left, and is called the (generalized) Verma module associated to $V$. We also observe that $M(V)\cong U_{-}\otimes_{\C}V$ as $\C$-vector spaces.

 If $V$ is finite-dimensional and irreducible, then $M(V)$ is called a minimal Verma module. We denote by $M(\lambda)$ the minimal Verma module $M(F(\lambda))$. A minimal Verma module is said to be non-degenerate if it is irreducible and degenerate if it is not irreducible.
\begin{definition}
We say that an element $w\in M(V)$ is homogeneous of degree $d$ if $w\in (U_-)_d\otimes V$.
\end{definition}
\begin{definition} A vector $w\in M(V)$ is called a singular vector if it satisfies the following conditions:
\begin{itemize}
\item[(i)] $x_i\partial_{i+1}w=0$ for every $i=1,\dots,4$;
\item[(ii)] $zw=0$ for every $z\in L_1$;
\item[(iii)] 
$w$ does not lie in $V$.
\end{itemize}
\end{definition}
We observe that the homogeneous components of positive degree of a singular vector are singular vectors. 
The same holds for its weight components. From now on we will thus assume that a singular vector is a homogeneous weight
vector unless otherwise specified.
Notice that if  condition (i) is satisfied then condition (ii) holds if $x_5d_{45}w=0$ since $x_5d_{45}$ is a lowest weight vector in $L_1$. 

\begin{proposition} \label{dege=sing}A minimal Verma module $M(V)$ is degenerate if and only if it contains a singular vector.
\end{proposition}
\begin{proof}
Let $w\in M(V)$ be a singular vector. We may assume that $w$ is homogeneous of degree $d>0$.  Hence the singular vector $w$ generates a submodule
of $M(V)$ which is proper since it is contained in $\oplus_{k\geq d} (U_-)_k\otimes V$.

On the other hand, if $M(V)$ is degenerate let us consider a proper non-zero submodule $W$ of $M(V)$. Let
$z\in W$ be a non-zero vector. By repeatedly applying $L_1$ to $z$ if necessary we can find a non-zero element $w\in W$
such that $L_1w=0$, since the action of $L_1$ lowers the degree of the homogeneous components of $z$ by 1.
We observe that $L_1$ vanishes on the $L_0$-module generated by $w$. Any highest weight vector in
such a module is a singular vector.    
\end{proof}

Degenerate Verma modules can also be described in terms of morphisms.
A linear map $\varphi: M(V)\rightarrow M(W)$ can always be associated to an element $\Phi\in U_{-}\otimes \Hom(V,W)$ as follows: for $u\in U_-$ and $v\in V$ we let
$$\varphi(u\otimes v)=u\Phi(v)$$
where, if $\Phi=\sum_iu_i\otimes \theta_i$ with $u_i\in U_-$,
$\theta_i\in \Hom(V,W),$ we let $\Phi(v)=\sum_iu_i\otimes \theta_i(v)$.
We will say that $\varphi$ (or $\Phi$) is a morphism of degree $d$ if $u_i\in (U_-)_d$ for every $i$.

\medskip

The following proposition characterizes morphisms between Verma modules.

\begin{proposition}\cite{R}\label{morphisms}
Let $\varphi: M(V)\rightarrow M(W)$ be the linear map associated with the element $\Phi\in U_{-}\otimes \Hom(V,W)$.
Then $\varphi$ is a morphism of $L$-modules if and only if the following conditions
hold:
\begin{itemize}
\item[(a)] $L_0.\Phi=0$;
\item[(b)] $t\Phi(v)=0$ for every $t\in L_1$ and for every $v\in V$.
\end{itemize} 
\end{proposition}
We observe that if $M(V)$ is a minimal Verma module and condition $(a)$ holds it is enough to verify  condition $(b)$ for an element $t$ generating
$L_1$ as an $L_0$-module and for $v$ a highest weight vector in $V$.

\begin{proposition}\label{morphism} Let $M(\mu)$ be a minimal Verma module. Then the following are equivalent:
\begin{itemize}
 \item[(a)] $M(\mu)$ is degenerate;
 \item[(b)] $M(\mu)$ contains a singular vector;
 \item[(c)] there exists a minimal Verma module $M(\lambda)$ and a morphism $\varphi:M(\lambda)\rightarrow M(\mu)$ of positive degree.
\end{itemize}
 \end{proposition}
 \begin{proof}
 We already know that condition (a) is equivalent to condition (b) by Proposition \ref{dege=sing}. 
Assume condition (c) holds: if $s\in F(\lambda)$ is a highest weight vector, then $\varphi(1\otimes s)$ is a singular vector in $M(\mu)$.

On the other hand, if $w$ is a singular vector in $M(\mu)$, we can define
$\varphi: M(\lambda(w))\rightarrow M(\mu)$ as the unique morphism of $L$-modules such that $\varphi(1\otimes s)=w$,  $s$ being a highest weight vector in $M(\lambda(w))$.
\end{proof}

\begin{remark}\label{dual}
Let $\varphi: M(V)\rightarrow M(W)$ be a linear map of degree $d$ associated to an element $\Phi\in U_-\otimes \Hom(V,W)$ that satisfies condition $(a)$ of Proposition \ref{morphisms}. Then there exists an $L_0$-morphism $\psi: (U_-)_d^*\rightarrow \Hom(V,W)$ such that
$\Phi= \sum_i u_i\otimes \psi(u_i^*)$ where $\{u_i, i\in I\}$ is any basis of $(U_-)_d$ and $\{u_i^*, i\in I\}$ is the corresponding dual basis.
\end{remark}

\begin{definition} Let $M(\mu)$ be a minimal Verma module and let $\pi: M(\mu)\rightarrow U_-\otimes F(\mu)_{\mu}$ be the natural projection,
$F(\mu)_{\mu}$ being the weight space of $F(\mu)$ of weight $\mu$. Given a singular vector $w\in M(\mu)$ we call $\pi(w)$ the leading term of $w$.
\end{definition}

\begin{proposition}\label{leading} If $w$ is a singular vector in $M(\mu)$ then:
\begin{itemize}
\item[(i)] $\pi(w)\neq 0$;
\item[(ii)] if two singular vectors in $M(\mu)$ have the same leading term then they coincide.
\end{itemize}
\end{proposition}
\begin{proof} If $w$ is a weight vector homogeneous of degree $d$ then we can write $w=\sum_iu_i\otimes v_i$ for some basis $\{u_i\}$ of $(U_-)_d$ consisting of weight vectors
 and $v_i\in F(\mu)_{\lambda_i}$ for some weight $\lambda_i$.
Let $\lambda_{i_0}$ be maximal in the dominance order such that $v_{i_0}\neq 0$. Then $v_{i_0}$ is a highest weight vector in $F(\mu)$. Indeed, for 
$r<s$ we have:
\[0=x_r\partial_s w=\sum_i[x_r\partial_s,u_i]\otimes v_i+\sum_i u_i\otimes x_r\partial_s.v_i. \] 
By the maximality of $\lambda_{i_0}$ it follows that $x_r\partial_s.v_{i_0}=0$.
 $(ii)$ follows from $(i)$.
\end{proof}

\section{Examples}\label{S4}
In this section we give some examples of singular vectors and the corresponding morphisms of Verma modules. These were described in \cite{R}. We will need the following technical result.
\begin{lemma}\label{esempi}
Let $\varphi: M(\lambda)\rightarrow M(W)$ be a morphism of Verma modules of degree one associated to
 $\Phi=\sum_{i<j} d_{ij}\otimes \theta_{ij}$  and  let $s$ be a highest weight vector in $F(\lambda)$.
 Let $\tilde W$ be an $L_0$-module containing $W$ and let $\tilde{\theta}_{ij}\in \Hom(F(\lambda),\tilde W)$ be such that
the map $(U_-)_1^*\rightarrow \Hom(F(\lambda),\tilde W)$ given by  $d_{ij}^*\mapsto \tilde{\theta}_{ij}$ is  well defined and $L_0$-equivariant. Then 
$\tilde{\theta}_{ij}(s)={\theta}_{ij}(s)$ implies
 $\tilde{\theta}_{ij}(v)={\theta}_{ij}(v)$ for all $v\in F(\lambda)$.
\end{lemma}
\begin{proof} It is enough to show that if $\tilde{\theta}_{ij}(v)={\theta}_{ij}(v)$
for some $v\in F(\lambda)$ and all $i\neq j$, then
$\theta_{ij}(x_h\partial_k.v)= \tilde{\theta}_{ij}(x_h\partial_k.v)$ for all $i\neq j$ and $h\neq k$. We have:
\begin{align*}
\tilde{\theta}_{ij}(x_h\partial_k.v)& =x_h\partial_k(\tilde{\theta}_{ij}(v))-(x_h\partial_k.\tilde{\theta}_{ij})(v)=
x_h\partial_k(\tilde{\theta}_{ij}(v))+\delta_{hi}\tilde{\theta}_{kj}(v)+\delta_{hj}\tilde{\theta}_{ik}(v)\\
&=x_h\partial_k({\theta}_{ij}(v))+\delta_{hi}{\theta}_{kj}(v)+\delta_{hj}{\theta}_{ik}(v)={\theta}_{ij}(x_h\partial_k.v)
\end{align*}
where we used Remark \ref{dual} in order to write the action of $L_0$ on the $\theta_{ij}$'s.
Namely, we have:
$$x_h\partial_k.\theta_{ij}=-\delta_{hi}\theta_{kj}-\delta_{hj}\theta_{ik}$$
where if $r>s$, $\theta_{rs}=-\theta_{sr}$.
\end{proof}

\begin{example}\label{nablaA}
Let us consider the Verma module $M(m,n,0,0)$. We first observe that $d_{12}\otimes x_1^mx_{12}^n$ is a singular vector
in $M(m,n,0,0)$. Indeed, for $i=1,\dots, 4$,
$$x_i\partial_{i+1}d_{12}\otimes x_1^mx_{12}^n=0;$$ 
besides,
$$x_5d_{45}d_{12}\otimes x_1^mx_{12}^n=x_5\partial_3\otimes x_1^mx_{12}^n=0.$$
By Proposition \ref{morphism} we can define a morphism of Verma modules $\nabla_A: M(m,n+1,0,0) \rightarrow M(m,n,0,0)$
by setting $\nabla_A(1\otimes s)=d_{12}\otimes x_1^mx_{12}^n$. By Lemma \ref{esempi} used with $\tilde W=\Sym^m(\C^5)\otimes \Sym^n(\inlinewedge^2\C^5)$ we have that $\nabla_A$ is associated to:
$$\sum_{i<j}d_{ij}\otimes \frac{\partial}{\partial x_{ij}}\in U_-\otimes \Hom(F(m,n+1,0,0),F(m,n,0,0)).$$
\end{example}

\begin{example}\label{nablaB}
Let us consider the Verma module $M(m,0,0,n+1)$. One can check that $\sum_{j=2}^5d_{1j}\otimes x_1^mx_j^*(x_{5}^*)^n$ is a singular vector
in $M(m,0,0,n+1)$, with leading term $d_{15}\otimes x_1^m(x_{5}^*)^{n+1}$. By Remark \ref{morphism} we can define a morphism of Verma modules $\nabla_B: M(m+1,0,0,n) \rightarrow M(m,0,0,n+1)$
by setting $\nabla_B(1\otimes s)=\sum_{j=2}^5d_{1j}\otimes x_1^mx_j^*(x_{5}^*)^n$. By Lemma \ref{esempi}, we have that $\nabla_B$ is associated to
$$\sum_{i<j}d_{ij}\otimes (x_i^*\partial_j-x_j^*\partial_i).$$
\end{example}

\begin{example} \label{nablaC} We shall now exhibit a singular vector in $M(0,0,m+1,n)$. To this aim it is convenient to think
of $F(0,0,m+1,n)$ as the dual $L_0$-module $F(n,m+1,0,0)^*$. We shall later investigate the role of duality between Verma modules in Section \ref{six}, where we will show, in particular, that the morphism we are going to construct can be seen in a certain sense as the dual of the morphism $\nabla_A$ defined in Example \ref{nablaA}.

Let us observe that the vector $\sum_{i<j}d_{ij}\otimes x_{ij}^*(x_{45}^*)^m(x_{5}^*)^n$ is a singular vector
in $M(F(n,m+1,0,0)^*)$ (with leading term $d_{45}\otimes (x_{45}^*)^{m+1}(x_{5}^*)^n$). Indeed, one immediately checks that $x_k\partial_{k+1}(\sum_{i<j}d_{ij}\otimes x_{ij}^*(x_{45}^*)^m(x_{5}^*)^n)=0$
for every $k=1,\dots, 4$. Besides, we have:
\begin{align*} x_5d_{45}&(\sum_{i<j}d_{ij}\otimes x_{ij}^*(x_{45}^*)^m(x_{5}^*)^n)\\
&= 
x_5\partial_3x_{12}^*(x_{45}^*)^m(x_5^*)^n-x_5\partial_2x_{13}^*(x_{45}^*)^m(x_5^*)^n+x_5\partial_1x_{23}^*(x_{45}^*)^m(x_5^*)^n\\
&=m(x_{45}^*)^{m-1}(x_5^*)^n(x_{12}^*x_{34}^*+x_{13}^*x_{42}^*+x_{14}^*x_{23}^*)
-n(x_{45}^*)^{m}(x_5^*)^{n-1}(x_{12}^*x_{3}^*+x_{23}^*x_{1}^*+x_{31}^*x_{2}^*)
&=0.
\end{align*}
Notice that, in fact,
\[ x_{ab}^*x_{cd}^*+x_{ac}^*x_{db}^*+x_{ad}^*x_{bc}^*=0\] 
 and
\[ x_{ab}^*x_{c}^*+x_{bc}^*x_{a}^*+x_{ca}^*x_{b}^*=0\] in $F(n,m+1,0,0)^*$
for all $a,b,c,d \in[5]$,
as one can check by applying these elements to the highest weight vector $x_1^nx_{12}^{m+1}$ in $F(n,m+1,0,0)$
and using the $L_0$-action.

  By Remark \ref{morphism} we can thus define a morphism of Verma modules $\nabla_C: M(0,0,m,n) \rightarrow M(F(n,m+1,0,0)^*)$
by setting $\nabla_C(1\otimes s)=\sum_{i<j}d_{ij}\otimes x_{ij}^*(x_{45}^*)^m(x_{5}^*)^n$. Once again, Lemma \ref{esempi} implies that the morphism $\nabla_C$ is associated to
$$\sum_{i<j}d_{ij}\otimes x_{ij}^*.$$
\end{example}
Examples \ref{nablaA}, \ref{nablaB} and \ref{nablaC} imply the following result.
\begin{proposition}
Let $m,n\geq 0$. Then $M(m,n,0,0)$, $M(m,0,0,n)$ and $M(0,0,m,n)$ are degenerate Verma modules.
\end{proposition}
Kac and Rudakov proposed the following conjecture \cite{KR}:
\begin{conjecture}\label{conjecture}
Let $a,b,c,d\geq 0$ be such that $M(a,b,c,d)$ is a degenerate Verma module. Then $a=b=0$ or $b=c=0$ or $c=d=0$.
\end{conjecture}
By Proposition \ref{morphism} a possible strategy to prove Conjecture \ref{conjecture} is to construct all possible morphisms between minimal Verma modules. One of
the main results of this paper is a complete classification of such morphisms of degree at most 3.
\begin{example}\label{exampledeg2}
The following are nonzero morphisms of degree 2:
\begin{itemize}
\item $\nabla_B\nabla_A:M(m,1,0,0)\rightarrow M(m-1,0,0,1)$;
\item $\nabla_C\nabla_B:M(1,0,0,n)\rightarrow M(0,0,1,n+1)$;
\item $\nabla_C\nabla_A:M(0,1,0,0)\rightarrow M(0,0,1,0)$;
\end{itemize}
Indeed, 
\[
\nabla_B\nabla_A(1\otimes x_1^mx_{12})=\nabla_B(d_{12}\otimes x_1^m)=-m\sum_{j>1}d_{12}d_{1j}\otimes x_1^{m-1}x_j^*\neq 0
\]
\[
\nabla_C\nabla_B(1\otimes x_1(x_5^*)^n)=\sum_{j>1}\sum_{h<k} d_{1j}d_{hk}\otimes x_{hk}^*x_j^*(x_5^*)^n\neq 0 
\]
\[
\nabla_C\nabla_A (1\otimes x_{12})=\sum_{i<j}d_{12}d_{ij}\otimes x_{ij}^*\neq 0.
\]
We observe that the leading terms of these singular vectors are $d_{12}d_{15}\otimes x_1^{m-1}x_5^*$,  $d_{15}d_{45}\otimes x_{45}^*(x_5^*)^{n+1}$
and $d_{12}d_{45}\otimes x_{45}^*$, respectively.
(We also observe that the other compositions $\nabla_A\nabla_B$, $\nabla_A\nabla_C$, $\nabla_B\nabla_C$ are not defined). Moreover, one can also verify that $\nabla_A^2=\nabla_B^2=\nabla_C^2=0$ whenever they are defined: this will also be a consequence of the general treatment of morphisms of degree 2 in Section \ref{eight}.
\end{example}

\begin{example}
\[
\nabla_C\nabla_B\nabla_A:M(1,1,0,0)\rightarrow M(0,0,1,1)
\]
is a nonzero morphism of degree 3. We have that $\nabla_C\nabla_B\nabla_A(x_1x_{12})=\sum_{j>1, k<l} d_{12}d_{1j}d_{kl}\otimes x_{j}^*x_{kl}^*$
is a  singular vector in $M(0,0,1,1)$ with leading term $d_{12}d_{15}d_{45}\otimes x_{45}^*x_5^*$ .
\end{example}
We will prove that the morphisms described in this section are all possible morphisms between minimal Verma modules of degree at most 3.

\section{Structure of $U_-$}\label{S5}
In order to classify morphisms between generalized Verma modules of a given degree we need to  better understand the structure of $U_-$ as an $L_0$-module. The main result of this section is the construction of an explicit linear basis of $U_-$ which realizes its structure of $L_0$-module in a combinatorial way. 

We recall that $(U_-)_d$ denotes the homogeneous component of $U_-$ of degree $d$. 
We let 
\[\mathcal I_d=\{ I=(I_1,\ldots,I_d):\, I_l=(i_l,j_l) \textrm{ with  $ 1\leq i_l,j_l\leq 5$ for every $l=1,\ldots,d$}\}.
\]
If $I=(I_1,\ldots,I_d)\in \mathcal I_d$  we let $d_I=d_{I_1}\cdots d_{I_d}\in (U_-)_d$, with $d_{I_l}=d_{i_l j_l}$. 

We set 
$[5]^k=\{(t_1,\dots, t_k)~|~ t_i\in [5]\}$ and 
for $T=(t_1,\dots, t_k)\in [5]^k$ we let
$\partial_T=\partial_{t_1}\dots\partial_{t_k}$.


We have that $(U_-)_d$ is spanned by all elements of the form $d_I$ as $I$ varies in $\mathcal I_d$. One can also consider the following filtration of subspaces of $(U_-)_d$: for all $k\leq d/2$ we let 
\[
(U_-)_{d,k}=\textrm{Span}\{\partial_Td_I:\, T\in \mathbb [5]^k,\, I\in \mathcal I_{d-2k}\}.
\]
We have the following chain of inclusions
\[
(U_-)_d=(U_-)_{d,0}\supseteq (U_-)_{d,1}\supseteq (U_-)_{d,2}\supseteq\cdots.
\]
We observe that for all $k\leq d/2$ the subspace $(U_-)_{d,k}$ is also an $L_0$-submodule of $(U_-)_d$ and so we have the following isomorphism of $L_0$-modules
\[
(U_-)_d\cong \bigoplus_{k\leq d/2} (U_-)_{d,k}/(U_-)_{d,k+1},
\]
where we let $(U_-)_{d,k}=0$ if $k>d/2$. For example, we have
\[
(U_-)_5\cong \frac{(U_-)_{5,0}}{(U_-)_{5,1}}\oplus \frac{(U_-)_{5,1}}{(U_-)_{5,2}}\oplus (U_-)_{5,2}.
\]
Moreover, one can check that there is an isomorphism of $L_0$-modules $\psi:(U_-)_{d,k}/(U_-)_{d,k+1} \rightarrow \Sym^k({\mathbb C^5}^*)\otimes {\inlinewedge}^{d-2k}(\inlinewedge^2\mathbb C^5)$: this isomorphism is simply given by extending multiplicatively the following formulas
\[
\psi(\partial_i)=x_i^*,\,\,\psi(d_{ij})=x_{ij}.
\]
 and so we have that
\[
(U_-)_d\cong\bigoplus_{k< d/2} \Sym^k({\mathbb C^5}^*)\otimes \displaywedge^{d-2k}({\displaywedge}^2\mathbb C^5)
\]
as $L_0$-modules. The main goal of this section is to  explicitly  construct such isomorphism. 

We need some further technical notation. If $1\leq i,j\leq 5$ we let $\overline{(i,j)}=(j,i)$. There is a natural action of $B_d$, the Weyl group of type $B$ and rank $d$,  on $\mathcal I_d$ that can be described in the following way.
If $w=(\eta_1\sigma_1,\ldots,\eta_d \sigma_d)\in B_d$, where $\sigma=(\sigma_1,\ldots,\sigma_d)$ is a permutation of $[d]$ and $\eta_j=\pm 1$ for all $j\in [d]$, we let
\[
 w(I)=J
\]
where
\[
 J_j=\begin{cases}
      I_{\sigma_j}& \textrm{if $\eta_j=1$}\\ \overline{I_{\sigma_j}}& \textrm{if $\eta_j=-1$}.
     \end{cases}
\]
The fact that this is a $B_d$-action is an easy verification and is left to the reader.

We let $\mathcal S_d$ be the set of subsets of $[d]$ of cardinality 2, so that $|\mathcal S_d|=\binom{d}{2}$.

Note that elements in $\mathcal I_d$ are ordered tuples of ordered pairs, while elements in $\mathcal S_d$ are unordered tuples of unordered pairs.

If $\{k,l\}\in \mathcal S_d$ and $I\in \mathcal I_d$ we let $t_{I_k,I_l}=t_{i_k,j_k, i_l,j_l}$ and $
 \varepsilon_{I_k,I_l}=\varepsilon_{i_k,j_k, i_l,j_l}
$ (see Section \ref{S1}).

Note that the definitions of $t_{I_k,I_l}$ and $ \varepsilon_{I_k,I_l}$ do not depend on the order of $k$ and $l$ but only on the set $\{k,l\}$. We also let
\[
 D_{\{k,l\}}(I)=\frac{1}{2}  (-1)^{l+k}  \varepsilon_{I_k,I_l}\partial_{t_{I_k,I_l}}\in (U_-)_2.                                                                                           \]
 For example, if $I=((1,2),(2,3),(3,5))\in \mathcal I_3$ then $D_{\{1,3\}}(I)=\frac{1}{2} (-1)^4 \varepsilon_{12354}\partial_4=-\frac{1}{2}\partial_4$.
 \begin{definition}
 A subset $S$ of $\mathcal S_d$ is \emph{self-intersection free} if its elements are pairwise disjoint.\end{definition}
 For example $S=\{\{1,3\},\{2,5\}, \{4,7\}\}$ is self-intersection free while $\{\{1,3\},\{2,5\}, \{3,7\}\}$ is not. We denote by $\textrm{SIF}_d$ the set of self-intersection free subsets of $\mathcal S_d$.
\begin{definition}
Let $\{k,l\}, \{h,m\}\in \mathcal S_d$ be disjoint. We say that $\{k,l\}$ and $\{h,m\}$ \emph{cross} if exactly one element in $\{k,l\}$ is between $h$ and $m$. If $S\in \textrm{SIF}_d$ we let the crossing number $c(S)$ of $S$ be the number of pairs of elements in $S$ that cross.
\end{definition}
For example, if $S=\{\{1,3\},\{2,5\}, \{4,7\}\}$ then $\{1,3\}$ and $\{2,5\}$ cross, $\{1,3\}$ and $\{4,7\}$ do not cross, and $\{2,5\}$ and $\{4,7\}$ cross, so the crossing number of $S$ is $c(S)=2$ (see Figure \ref{figcross} for a graphical interpretation).
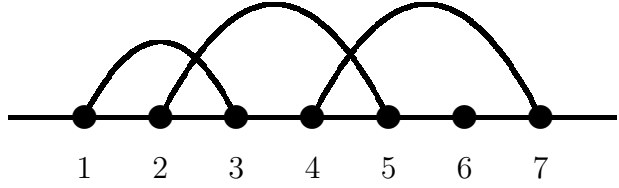
\begin{figure}
 \setlength{\unitlength}{1cm}
\begin{center}
\begin{picture}(8,6)
\thicklines

\linethickness{0.5mm}
\put(0,1){\line(1,0){8}}
\put(1,1){\circle*{0.3}}
\put(2,1){\circle*{0.3}}
\put(3,1){\circle*{0.3}}
\put(4,1){\circle*{0.3}}
\put(5,1){\circle*{0.3}}
\put(6,1){\circle*{0.3}}
\put(7,1){\circle*{0.3}}
\put(0.9,0.2){$1$}
\put(1.9,0.2){$2$}
\put(2.9,0.2){$3$}
\put(3.9,0.2){$4$}
\put(4.9,0.2){$5$}
\put(5.9,0.2){$6$}
\put(6.9,0.2){$7$}
\qbezier(1,1)(2,3)(3,1)
\qbezier(2,1)(3.5,4)(5,1)
\qbezier(4,1)(5.5,4)(7,1)

\end{picture}
\end{center}
\caption{A graphical interpretation of the crossing number}\label{figcross}
\end{figure}
\begin{definition}
 Let $S=\{S_1,\ldots,S_r\}\in \textrm{SIF}_d$. We let
 \[
  D_S(I)=\prod_{j=1}^rD_{S_j}(I)\in (U_-)_{2r}
 \]
if $r\geq 2$ and $D_{\emptyset}(I)=1$ (note that the order of multiplication is irrelevant as the elements $D_{S_j}(I)$ commute among themselves).
\end{definition}
\begin{definition}
 For $I=(I_1,\ldots,I_d)\in \mathcal I_d$ and $S=\{S_1,\ldots,S_r\}\in \textrm{SIF}_d$ we let $C_S(I)\in \mathcal I_{d-2r}$ be obtained from $I$ by removing all $I_j$ such that $j\in S_k$ for some $k\in [r]$. 
\end{definition}
For example, if $d=7$ and $S=\{\{1,4\},\{2,7\}\}$ then $C_S(I)=(I_3,I_5,I_6)$.
We are now ready to give the main definition of this section.

\begin{definition}\label{defomega}
 For all $I\in \mathcal I_d$ we let
 \[
  \omega_I=\sum_{S\in \textrm{SIF}_d}(-1)^{c(S)}D_S(I)\,d_{C_S(I)}\in (U_-)_d.
 \]

\end{definition}

For example, if $I=(21,13,45,25)\in \mathcal I_4$ we have
\begin{itemize}
 \item $D_{\emptyset}(I)=1$;
 \item $D_{\{1,3\}}(I)=-\frac{1}{2}\partial_3$;
 \item $D_{\{2,3\}}(I)= +\frac{1}{2}\partial_2$;
 \item $D_{\{2,4\}}(I)=+\frac{1}{2}\partial_4$;
 \item $D_{\{1,3\},\{2,4\}}(I)=D_{\{1,3\}}(I)D_{\{2,4\}}(I)=-\frac{1}{4}\partial_3\partial_4$
\end{itemize}
and all other $D_S(I)$ vanish. We also have, $c(\{\{1,3\},\{2,4\}\})=1$ so
\[
 \omega_I=d_I-\frac{1}{2}\partial_3d_{13}d_{25}+\frac{1}{2}\partial_2d_{21}d_{25}+\frac{1}{2}\partial_4 d_{21}d_{45}+\frac{1}{4}\partial_3\partial_4.
\]

\begin{proposition}
 For all $I\in \mathcal I_d$ and all $g\in B_d$ we have
 \[
  \omega_{g(I)}=(-1)^{\ell(g)}\omega_I,
 \]
where $\ell(g)$ is the length of $g$ with respect to the Coxeter generators $\{s_0,s_1,s_2,\ldots,s_{d-1}\}$, with $s_0=(-1,2,3,\ldots,d)$ and $s_1,\ldots,s_{d-1}$  the usual simple transpositions.
\end{proposition}
\begin{proof}
 It is enough to verify the statement for $g\in \{s_0,\ldots,s_{d-1}\}$. If $g=s_0$ we have, for all $k,l$, $1\leq k,l\leq d$:
 \begin{itemize}
  \item $\varepsilon_{s_0(I)_k,s_0(I)_l}=(-1)^{\chi_{1\in \{k,l\}}}\varepsilon_{I_k,I_l}$;
  \item $t_{s_0(I)_k,s_0(I)_l}=t_{I_k,I_l}$;
 \end{itemize}
 hence $D_S(s_0(I))=(-1)^{\chi_{1\in S}}D_S(I)$ while $d_{\mathcal C_{S}(s_0(I))}=(-1)^{\chi_{1\notin S}}d_{\mathcal C_{S}(I)}$, and
therefore we have
   \begin{align*} \omega_{s_0(I)}&=  \sum_{S\in SIF_d}(-1)^{c(S)}D_S(s_0(I))d_{C_S(s_0(I))}\\
 &= \sum_{S\in SIF_d}(-1)^{c(S)} (-1)^{\chi_{1\in S}}D_S(I) (-1)^{\chi_{1\notin S}}d_{\mathcal C_{S}}(I)\\
 &=-\omega_I.
  \end{align*}
  Now let $h\in \{1,\ldots,d-1\}$ and, for notational convenience, let $\sigma=s_h$. We have:
  \begin{itemize}
   \item $\varepsilon_{{\sigma}(I)_k,\sigma(I)_l}=\varepsilon_{I_{\sigma(k)},I_{\sigma(l)}}$;
   \item $t_{{\sigma}(I)_k,\sigma(I)_l}=t_{I_{\sigma(k)},I_{\sigma(l)}}$;
   \item $(-1)^{k+l}=(-1)^{{\sigma}(k)+{\sigma}(l)+\chi_{h\in \{k,l\}}+\chi_{h+1\in \{k,l\}}}$
   \end{itemize}
   hence $D_S({\sigma}(I))=(-1)^{\chi_{h\in S}+\chi_{h+1\in S}}D_{{\sigma}(S)}(I)$, where ``$h\in S$'' means that $h$ belongs to some element of $S$. We also observe that
   \[
    (-1)^{c(S)}=(-1)^{c({\sigma}(S))}(-1)^{\chi_{h\in S,\,h+1\in S,\,\{h,h+1\}\notin S}}
   \]
i.e. the parity of the crossing number of $S$ is opposite to the parity of the crossing number of ${\sigma}(S)$ precisely if $h$ and $h+1$ belong to two distinct elements of $S$.
   Moreover we observe that $d_{C_S({\sigma}(I))}=d_{C_{{\sigma}(S)}(I)}$ if $h$ or $h+1$ belong to $S$. If $h,h+1$ do not belong to $S$ we have
   \[
    d_{C_S({\sigma}(I))}=-d_{C_{{\sigma}(S)}(I)}-2D_{\{h,h+1\}}(I)d_{C_{{\sigma}(\tilde S)}(I)}
   \]
   where $\tilde S$ is obtained from $S$ by adding the pair $\{h,h+1\}$. We are now ready to compute $\omega_{{\sigma}(I)}$. We have
   \begin{align*}
   \omega_{{\sigma}(I)}&=\sum_{S\in SIF_d}(-1)^{c(S)}D_S({\sigma}(I))d_{C_S({\sigma}(I))}\\
   &= \sum_{S\ni h\,or\, S\ni h+1 \,but\, S\not \ni \{h,h+1\}} (-1)^{c(S)}D_S({\sigma}(I))d_{C_S({\sigma}(I))}\\
   &\hspace{3mm}+\sum_{S \not \ni h,h+1}\Big((-1)^{c(S)}D_S({\sigma}(I))d_{C_S(s_0(I))}+(-1)^{c(\tilde S)}D_{\tilde S}({\sigma}(I))d_{C_ {\tilde S}(\sigma(I))}\Big)\\
   &= \sum_{S\ni h\,or\, S\ni h+1 \,but\, S\not \ni \{h,h+1\}} (-1)^{\chi_{h\in S,\,h+1 \in S}}(-1)^{c({\sigma}(S))}(-1)^{\chi_{h\in S}+\chi_{h+1\in S}} D_{{\sigma}(S)}(I)d_{C_{{\sigma}(S)}(I)}\\
   &\hspace{3mm}+ \sum_{S \not \ni h,h+1}\Big( (-1)^{c({\sigma}(S))}D_{{\sigma}(S)}(I)(-d_{C_{{\sigma}(S)}(I)}-2D_{\{h,h+1\}}(I)d_{C_{{\sigma}(\tilde S)}(I)})+(-1)^{c({\sigma}(\tilde S))}D_{{\sigma}(\tilde S)}(I)d_{C_{\sigma(\tilde S)}(I)}\Big)\\
   &=-\sum_{S\ni h\,or\, S\ni h+1 \,but\, S\not \ni \{h,h+1\}} (-1)^{c({\sigma}(S))} D_{{\sigma}(S)}(I)d_{C_{{\sigma}(S)}(I)}-\sum_{S \not \ni h,h+1}(-1)^{c({\sigma}(S))}D_{{\sigma}(S)}(I)d_{C_{{\sigma}(S)}(I)}\\
   &\hspace{3mm}+\sum  _{S \not \ni h,h+1}\Big( (-2)(-1)^{c({\sigma}(S))}D_{{\sigma}(S)}(I)D_{\{h,h+1\}}(I)d_{C_{{\sigma}(\tilde S)}(I)})+(-1)^{c({\sigma}(\tilde S))}D_{{\sigma}(\tilde S)}(I)d_{C_{\sigma(\tilde S)}(I)}\Big)\\
   &=-\sum_{S\in SIF_d}(-1)^{c(\sigma(S))}D_{\sigma(S)}(I)\, d_{C_{\sigma(S)}(I)}\\
   &=-\omega_I
\end{align*}
 where we used that $D_{{\sigma}(S)}(I)D_{h,h+1}(I)=D_{\sigma(\tilde S)}(I)$ and $(-1)^{c({\sigma}(S))}=(-1)^{c({\sigma}(\tilde S))}$.
\end{proof}
\begin{corollary}
 If $I=(I_1,\ldots,I_d)$ is such that $I_j= I_k$ for some $j<k$, then $\omega_I=0$; if $I_j=\overline I_k$ for some $j\leq k$, then $\omega_I=0$.
\end{corollary}

Now we want to study the action of $L_0$ on the elements $\omega_I$. 
If $I=(I_1,\ldots,I_d)$ and $r$ appears once in $I_b$ for some $b$ we let $I^{b,s,r}$ be the sequence obtained from $I$ by substituting the letter $r$ in $I_b$ by $s$. We want to prove the following
\begin{theorem}\label{act} Let $I\in \mathcal I_d$ and $r,s\in [5]$, $r\neq s$.
Assume that the letter $r$ appears in $I_1,\ldots,I_c$, once in each pair,  and does not appear in $I_{c+1},\ldots,I_{d}$. Then
\[
 x_s\partial _r. \omega_I=\sum_{b=1}^c \omega_{I^{b,s,r}}.\]
\end{theorem}
\begin{proof} For notational convenience, since $r$ and $s$ are fixed in this proof, we simply let $I^b=I^{b,s,r}$ for all $1\leq b\leq c$.
 We start by calculating the left-hand side. We have
 \begin{align*}
  x_s\partial _r. \omega_I&= x_s \partial _r. \sum_S(-1)^{c(S)} D_S(I) d_{C_S(I)}.
  \end{align*}
  Now we observe that $x_s\partial_r. D_{\{k,l\}}(I)$ is non zero if and only if $I_k$ and $I_l$ have the four indices distinct from $s$, hence $k$ and $l$ cannot be both less than or equal to $c$
	or both strictly greater than $c$. We then assume that $k\leq c$ and $l>c$; in this case we have 
  \[
  x_s\partial_r. D_{\{k,l\}}(S)=x_s \partial _r. \Big(\frac{1}{2}(-1)^{k+l}\varepsilon_{I_k,I_l}\partial_{t_{I_k,I_l}}\Big)=\frac{1}{2}(-1)^{k+l+1}\varepsilon_{I_k,I_l}\partial_r.
  \]
So we have
  \begin{align*}
 x_s\partial _r \omega_I=& \sum_{k\leq c<l,\,s\notin I_k,\,s\notin I_l}\frac{1}{2}(-1)^{k+l+1}\varepsilon_{I_k,I_l}\partial_r\sum_{S\not \ni k,l}(-1)^{c(S)}D_S(I)d_{C_{S\cup \{k,l\}}(I)}+\\
&+\sum_S (-1)^{c(S)}D_S(I) \sum_{b\leq c,\,b\notin S}d_{C_S(I^b)}.
 \end{align*}
 
 Now we compute the right-hand side:
 \begin{align*}
 \sum_{b\leq c} \omega_{I^{b}}=\sum_{b\leq c} \sum_{S}(-1)^{c(S)}D_S(I^b)d_{C_S(I^b)}.
 \end{align*}
Now we observe that if $b\notin S$ we have $D_S(I^b)=D_S(I)$ and so we reduce to prove the following:
\[
 \sum_{k\leq c<l,\,s\notin I_k,\,s\notin I_l}\frac{1}{2}(-1)^{k+l+1}\varepsilon_{I_k,I_l}\partial_r\sum_{S\not \ni k,l}(-1)^{c(S)}D_S(I)d_{C_{S\cup \{k,l\}}(I)}=\sum_{S,b: b\leq c,\,b\in S} (-1)^{c(S)}D_S(I^b)d_{C_S(I^b)}
\]
We notice that if $\{b,b'\}\in S$ with both $b,b'\leq c$ then $D_S(I^b)=-D_S(I^{b'})$ hence we reduce to prove that
\begin{align*}
 \sum_{k\leq c<l,\,s\notin I_k,\,s\notin I_l}\frac{1}{2}(-1)^{k+l+1}\varepsilon_{I_k,I_l}\partial_r\sum_{S\not \ni k,l}&(-1)^{c(S)}D_S(I)d_{C_{S\cup \{k,l\}}(I)}
=\\
&=\sum_{b\leq c<l}\sum_ {S:\, S\not\ni b,l} (-1)^{c(S)}D_{\{b,l\}}(I^b)D_S(I^b)d_{C_{S\cup\{b,l\}}(I^b)}.
\end{align*}

Finally, in order to prove this last equation we observe that if $b\leq c<l$ then  $D_{\{b,l\}}(I^b)$ is nonzero only if $s\notin I_b,I_l$, that in this case $\varepsilon_{(I^b)_b,(I^b)_l}=-\varepsilon_{I_b,I_l}$, that $D_{\{b,l\}}(I^b)= -\frac{1}{2}(-1)^{b+l}\varepsilon_{I_b,I_l}\partial _r$ and that $d_{\mathcal C_{S\cup \{b,l\}}(I^{b})}=d_{\mathcal C_{S\cup \{b,l\}}(I)}$.
The proof is complete.
\end{proof}

\medskip

If $I=(I_1,\ldots,I_d)$ with $I_k=(i_k,j_k)$ we let 
\[
 D_{s\rightarrow r}(\omega _I)=\delta_{r,i_{1}}\omega_{((s,j_{1}),I_2,\ldots,I_d)}+\delta_{r,j_{1}}\omega_{((i_{1},s),I_2,\ldots,I_d)}+\delta_{r,i_{2}}\omega_{(I_1,(s,j_{2}),I_3,\ldots,I_d)}+\cdots+\delta_{r,j_{d}}\omega_{(I_1,\ldots,I_{d-1},(i_{d},s))}.
\]

\begin{corollary}\label{omegaction}
 Let $I=(I_1,\ldots,I_d)$ be arbitrary. Then
 \[
  x_s\partial_r \omega_I=  D_{s\rightarrow r}(\omega _I).
 \]
\end{corollary}
\begin{proof}
 If there exists $k$ such that $i_{k}=j_{k}$ then $\omega_I=0$ and clearly also $D_{s\rightarrow r}(\omega _I)=0$ since all summands in the definition above vanish except possibly two of them which cancel out.
 If such $k$ does not exist let $w\in B_d$ be such that $J=w(I)$ satisfies the following property: there exists $0\leq c\leq d$ such that $r$ appears in $J_1,\ldots,J_c$ and does not appear in $J_{c+1},\ldots,J_d$. By Theorem \ref{act} we know that the result holds for $J$ hence the result follows since $D_{s\rightarrow r}$ commutes with the action of $B_d$ (we leave this to the reader).
\end{proof}
\begin{corollary}\label{basi}
The map 
\[
\varphi:\bigoplus_{k}\Sym^k({\mathbb C^5}^*)\otimes \displaywedge^{d-2k}(\displaywedge^2 \mathbb C^5)\rightarrow (U_-)_d
\]
given by
\[
\varphi(x_{t_1}^*\cdots x_{t_k}^*\otimes x_{i_1 j_1}\wedge \cdots\wedge x_{i_{d-2k} j_{d-2k}})=\partial_{t_1}\cdots \partial_{t_k} \omega_{(i_1,j_1),\ldots,(i_{d-2k},j_{d-2k})}
\]
for all $k\leq d/2$ and  $t_1,\dots,t_k, i_1, j_1,\dots, i_{d-2k}, j_{d-2k}\in [5]$ is an isomorphism of $L_0$-modules, hence the set
$$\bigcup_{k\leq d/2}\{\partial_T\omega_I~|~ T=(t_1,\dots,t_k)\in[5]^k, t_1<\dots<t_k, I\in{\mathcal{I}}_{d-2k}/B_{d-2k}\}$$
is a basis of $(U_-)_d$.
\end{corollary}
\section{Properties of the dominance order}\label{S6}
In this section we establish simple combinatorial criteria to determine whether the weights of vectors in $U_-$ and $(U_-)^*$ are comparable.
\begin{remark}\label{thetaction} If $\varphi:M(V)\rightarrow M(W)$ is a linear map of degree $d$ which satisfies condition (a) of Proposition \ref{morphisms} let $\psi:(U_-)_d^*\rightarrow \Hom(V,W)$  be as in Remark \ref{dual}. By Corollary \ref{basi} we can identify $(U_-)_d^*$ with $\bigoplus_{k}\Sym^k({\mathbb \C^5})\otimes \inlinewedge^{d-2k}(\inlinewedge^2 \mathbb (\C^5)^*)$ and we let for all $T=(t_1,\dots,t_k)\in [5]^k$ and
$I=(I_1,\dots, I_{d-2k})\in \mathcal{I}_{d-2k}$ with $I_h=(i_h,j_h)$, 
\[\theta_I^T=\psi(x_{t_1}\cdots x_{t_k}\otimes x_{i_1 j_1}^*\wedge \cdots\wedge x_{i_{d-2k} j_{d-2k}}^*).
\]
We observe that $\theta_{g(I)}^T=(-1)^{\ell(g)}\theta_I^T$ for every $g\in B_{d-2k}$ hence $\partial_T\omega_I\otimes \theta_I^T$ is invariant
with respect to the action of $B_{d-2k}$ on $I$. We can thus write
$$\Phi=\sum_{\substack{T=(t_1,\ldots,t_k):\\1\leq t_1\leq \cdots \leq t_k\leq 5}}\sum_{I\in {\mathcal I}_{d-2k}/B_{d-2k}}\partial_T\omega_I\otimes\theta_{I}^T.$$
Moreover, we have:
\[x_s\partial_r.\theta_I^T=x_s\partial_r.\theta_{I_1,\dots,I_{d-2k}}^{t_1,\dots,t_k}=\sum_{h=1}^k\Delta_{s\rightarrow r}^{h}\theta_{I}^T
-\sum_{l=1}^{d-2k}D_{r\rightarrow s}^{l}\theta_I^T\]
where $\Delta_{s\rightarrow r}^h(\theta_I^T)=\delta_{r,t_h}\theta_I^{t_1,\dots,t_{h-1},s,t_{h+1},\dots, t_k}$ and 
\[D_{r\rightarrow s}^l(\theta_I^T)=\delta_{i_l,s}\theta^T_{I_1,\dots,I_{l-1},(r,j_l),I_{l+1},\dots, I_{d-2k}}+\delta_{s,j_l}\theta^T_{I_1,\dots,I_{l-1},(i_l,r),I_{l+1},\dots, I_{d-2k}}.\]
\end{remark}

We now study the dominance order on the weights of the elements $d_I$, $\omega_I$ and $\theta_I^T$. This will turn out to play a fundamental role in the study of morphisms of
Verma modules. 

We observe that $d_{kl}$ is a weight vector for $L_0$. Indeed we have:
\[[h_{ij},d_{kl}]=(\delta_{i,k}+\delta_{i,l}-\delta_{j,k}-\delta_{j,l})d_{k,l}
\]
and so $\lambda_{ij}(d_{kl})$ is the number of occurrences of $i$  minus the number of occurrences of $j$ in $\{k,l\}$. 
If $I=(i_1,\ldots,i_d)$ is a sequence of integers and we let 
\[
m_k(I)=|\{s\in [d]:\, i_s=k\}|\
\]
be the multiplicity of $k$ in $I$, 
 we have
\[
\lambda_{ij}(d_{kl})=m_i(k,l)-m_j(k,l).
\]
More generally, if $I=\{i_{1},j_{1},\ldots,i_{d},j_{d}\}$ and $d_I=d_{i_{1}j_{1}}\cdots d_{i_{d}j_{d}}$ we have
\[
\lambda_{ij}(d_I)=m_i(I)-m_j(I).
\]

In order to understand when the weights of $d_I$ and $d_K$ are comparable in the dominance order, we first observe that the weight of $d_I$ does not depend on the order of its entries.
If $I=(i_1,\ldots,i_{2d})$ we let $I_o=(i'_1,\ldots,i'_{2d})$ be the non decreasing reordering of $I$. We write $I\leq K$ if $i'_1\leq k'_1,\,\ldots, i'_{2d}\leq k'_{2d}$ and $I<K$ if $I\leq K$ and at least one of the previous inequalities is strict (notice that this is different that requiring $I\neq K$). 
\begin{proposition}\label{dominance}For all $I,K\in \mathcal I_d$ we have $\lambda(d_I)\geq \lambda (d_K)$ if and only if $I\leq K$.
\end{proposition}
\begin{proof}
We can assume that $I=(i_1,\ldots,i_{2d})$ and $K=(k_1,\ldots,k_{2d})$ are such that $I=I_o$ and $K=K_o$.\
We express the difference of the weights as a linear combination of roots. First assume that all entries of $I$ and $K$ coincide except in position $r$ and that $i_r=h$ and $k_r={h+1}$.
We have $m_l(I)=m_l(K)$ for all $l\neq h,h+1$, $m_h(I)=m_h(K)+1$ and $m_{h+1}(I)=m_{h+1}(K)-1$. Therefore $\lambda_{l,l+1}(d_I)=\lambda_{l,l+1}(d_K)$ for all $l\neq h-1,h,h+1$, $\lambda_{h-1,h}(d_I)=\lambda_{h-1,h}(d_K)+1$, (if $h\neq 1$), $\lambda_{h,h+1}(d_I)=\lambda_{h,h+1}(d_K)-2$ and $\lambda_{h+1,h+2}(d_I)=\lambda_{h+1,h+2}(d_K)+1$ (if $h\neq 4$). Therefore
\[
\lambda(d_I)-\lambda(d_K)=\alpha_{h,h+1}.
\]
From this we can deduce that 
\[
\lambda(d_I)-\lambda(d_K)=\alpha_{i_1,k_1}+\alpha_{i_2,k_2}+\cdots+\alpha_{i_{2d},k_{2d}}.
\]
In particular, if $i_1\leq k_1,\ldots,i_{2d}\leq k_{2d}$ then $\lambda(d_I)\geq \lambda(d_K)$. Now we assume that the inequalities $i_1\leq k_1,\ldots,i_{2d}\leq k_{2d}$ are not all satisfied and we let $r$ be minimum such that $i_r>k_r$. If we express $\lambda(d_I)-\lambda(d_K)$ as a linear combination of the simple roots then $\alpha_{k_r,k_{r+1}}$ necessarily appears with a negative coefficient and we are done.
\end{proof}
\begin{corollary}\label{dominancetheta} For all $I,K\in \mathcal I_d$ and all $T,R\in [5]^k$ we have:
\begin{itemize}
\item[(i)] $\lambda(\theta^T_I)\leq \lambda (\theta_K^T)$ if and only if $I\leq K$;
\item[(ii)] $\lambda(\theta^T_I)\geq \lambda (\theta_I^R)$ if and only if $T\leq R$.
\end{itemize}
\end{corollary}
\begin{proof} In order to prove $(i)$ it is sufficient to notice that $\lambda(\theta_I^T)=-(\lambda(\partial_T\omega_I))
=-\lambda(\partial_T)-\lambda(d_I)$ and then use Proposition \ref{dominance}.

In order to prove $(ii)$ it is convenient to introduce the following notation. For $t\in[5]$ we let $t^{(1)}<t^{(2)}< t^{(3)} <t^{(4)}$
such that $\{t, t^{(1)}, t^{(2)}, t^{(3)} ,t^{(4)}\}=[5]$ and, 
for
$T=(t_1,\dots, t_k)\in[5]^k$, $T^c=(t_1^{(1)} t_1^{(2)}, t_1^{(3)} t_1^{(4)}, \dots, t_k^{(1)} t_k^{(2)}, t_k^{(3)} t_k^{(4)})\in\mathcal{I}_{2k}$. 
Then it is enough to notice that $\lambda(\partial_T)=\lambda(d_{T^c})$ and that $T\leq R$ if and only if $T^c\geq R^c$. Then one can use $(i)$.
\end{proof}
\section{Duality}\label{six}
Consider a morphism  $\varphi:M(V)\rightarrow M(W)$ of generalized Verma modules of degree $d$ associated to an element $\Phi\in (U_-)_d\otimes \Hom(V,W)$. We ask the natural question: does it exist a ``related'' morphism $\psi:M(W^*)\rightarrow M(V^*)$ of the same degree $d$? The first natural candidate to look at is the following: if $\Phi=\sum_i u_i\otimes \theta_i$, where $\{u_i~|~i\in I\}$ is any basis of $(U_-)_d$ and  $\theta_i\in \Hom(V,W)$ then we can consider the linear map $\psi:M(W^*)\rightarrow M(V^*)$ associated to $\Psi=\sum_i u_i \otimes \theta_i^*$, where, for all $\theta\in \Hom(V,W)$ we denote by $\theta^*\in \Hom(W^*,V^*)$ the pull-back of $\theta$ given by $\theta^*(f)=f\circ \theta$ for all $f\in W^*$. One can easily check that the map $\psi$ does not depend on the chosen basis $\{u_i~|~i\in I\}$ of $(U_-)_d$. It turns out that for $d=1$ the map $\psi$ is also a morphism of $L$-modules, but this is not the case in general if the degree $d$ is at least 2.

In this section we develop some tools which will allow us to construct a morphism of $L$-modules $\psi:M(W^*)\rightarrow M(V^*)$ starting from a morphism $\varphi:M(V)\rightarrow M(W)$ of degree at most 3 and we conjecture that our construction provides such morphism in all degrees.

The main result that we will need is the following. 
\begin{proposition}\label{lemdual}
Let $\theta_1,\ldots,\theta_r, \sigma_1,\ldots, \sigma_s\in \Hom(V,W)$ for some $L_0$-modules $V$, $W$, and let $z_1,\ldots,z_t\in L_0$. Let $a_i, b_{j,k}\in \mathbb C$ be such that
\[
 \sum_i a_i \theta_i(v)+\sum_{j,k}b_{j,k}\big(z_k.(\sigma_j(v))+\sigma_j(z_k.v)\big)=0\in W
\]
for all $v\in V$. Then
\[
 \sum_i a_i \theta^*_i(f)+\sum_{j,k}b_{j,k}\big(z_k.((-\sigma_j^*)(f))+(-\sigma_j^*)(z_k.f)\big)=0\in V^*
\]
for all $f\in W^*$.
\end{proposition}
\begin{proof}
For all $v\in V$ we have
\begin{align*}
  \Big( \sum_i a_i &\theta^*_i(f)+\sum_{j,k}b_{j,k}\big(z_k.((-\sigma_j^*)(f))+(-\sigma_j^*)(z_k.f)\big)\Big)(v)\\
  &=\sum_i a_i f(\theta_i(v))+\sum_{j,k}b_{j,k}\big(\sigma_j^*(f))(z_k.v)+(z_k.f)(-\sigma_j(v))\big)\\
  &=\sum_i a_i f(\theta_i(v))+\sum_{j,k}b_{j,k}\big(f(\sigma_j(z_k.v))+f(z_k.(\sigma_j(v))\big)\\
  &=f\Big( \sum_i a_i \theta_i(v)+\sum_{j,k}b_{j,k}\big(\sigma_j(z_k.v)+z_k.(\sigma_j(v)\big) \Big)\\
  &=0.
\end{align*}

\end{proof}

\begin{remark}\label{lemdual2}
We will use Proposition \ref{lemdual} also in the following equivalent formulation: let $\theta_1,\ldots,\theta_r$, $\sigma_1,\ldots, \sigma_s\in \Hom(V,W)$ for some $L_0$-modules $V$, $W$ and $z_1,\ldots,z_t\in L_0$. Let $a_i, b_{j,k}\in \mathbb C$ be such that
\[
 \sum_i a_i \theta_i(v)+\sum_{j,k}b_{j,k}\big(2z_k.(\sigma_j(v))-(z_k.\sigma_j)(v)\big)=0\in W
\]
for all $v\in V$. Then
\[
 \sum_i a_i \theta^*_i(f)+\sum_{j,k}b_{j,k}\big(2z_k.((-\sigma_j^*)(f))-(z_k.(-\sigma_j^*))(f)\big)=0\in V^*
\]
for all $f\in W^*$.
\end{remark}

\begin{conjecture}\label{conjdual}
 Let $\varphi:M(V)\rightarrow M(W)$ be a morphism of degree $d$ associated to $\Phi:=\sum_{T,I}{\partial_T\omega_I}\otimes \theta^T_I$ for some $\theta^T_I\in \Hom(V,W)$. Then the linear map $\psi:M(W^*)\rightarrow M(V^*)$ associated to $\Psi:=\sum_{T,I}\partial_T\omega_I\otimes (-1)^{\ell(T)}(\theta^T_I)^*$ is also a morphism of Verma modules, where
 if $T\in[5]^k$, we let $\ell(T)=k$.
\end{conjecture}
In the following sections we will verify Conjecture \ref{conjdual} for morphisms of degree at most 3 as a straightforward application of Proposition \ref{lemdual}.

\begin{definition}
 Let $\varphi:M(\lambda)\rightarrow M(\mu)$ be a morphism of Verma modules. The weight $\mu-\lambda$ is called the leading weight of $\varphi$.
\end{definition}
The reason of the terminology in the previous definition is motivated by the following observation.
\begin{remark}
Let $\varphi:M(\lambda)\rightarrow M(\mu)$ be a morphism of Verma modules of leading weight $\nu$. If $\varphi$ is associated to $\Phi=\sum_{i}u_i\otimes \theta_i$, where $\{u_i ~|~ i\in I\}$ is a basis of $(U_-)_d$ consisting of weight vectors, let $\theta_{i_0}$ be  of maximal weight such that $\theta_{i_0}(s)\neq 0$ for a highest weight vector $s\in F(\lambda)$. Then $\theta_{i_0}(s)$ is a highest weight vector in $F(\mu)$ and so the weight of $\theta_{i_0}$ is the leading weight of $\varphi$. Therefore if $\varphi$ has leading weight $\nu$ the leading
term of the singular vector $\varphi(1\otimes s)$ is 
$$\sum_{i: \lambda(\theta_i)=\nu}u_i\otimes \theta_i(s).$$
We also say that $\theta\in \Hom(V,W)$ has the leading weight of $\varphi$ if $\theta(s)\neq 0$ and the weight of $\theta$ is $\nu$. A general strategy to study a morphism $\varphi:M(V)\rightarrow M(W)$ is to understand elements $\theta\in \Hom(V,W)$ which have the leading weight of $\varphi$; in particular we will often show that there is no such morphism by showing that there is no $\theta\in \Hom(V,W)$ that may possibly have the leading weight of a morphism.
\end{remark}

Whenever Conjecture \ref{conjdual} holds the next result allows us to simplify the classification of morphisms.

\begin{remark}
 Let $\varphi:M(V)\rightarrow M(W)$ and $\psi:M(W^*)\rightarrow M(V^*)$ be morphisms of Verma modules and let $\nu=(a,b,c,d)$ be the leading weight of $\varphi$. Then the leading weight of $\psi$ is $-\nu^*=-(d,c,b,a)$.
\end{remark}

\section{Morphisms of degree one}\label{S7}
In this section we classify morphisms of degree one between generalized Verma modules, slightly simplifying Rudakov's argument \cite{R}.

We let $C(a,b,c)$ be the set of cyclic permutations of $a,b,c$, i.e., $C(a,b,c)=\{(a,b,c), (b,c,a)$, $(c,a,b)\}$. 

\begin{theorem}
Let $\varphi:M(V)\rightarrow M(W)$ be a linear map of degree one associated to
$$\Phi=\sum_{I\in{\mathcal I}_1/B_1}\omega_I\otimes \theta_I$$
such that $L_0.\Phi=0$. Then $\varphi$ is a morphism of Verma modules if and only if 
for all distinct $a,b,c,p\in [5]$ and for all $v\in V$ we have
\begin{equation}
\sum_{(\alpha,\beta,\gamma)\in C(a,b,c)}
x_p\partial_\gamma. (\theta_{\alpha\beta}(v))=0.
\label{equazione1}
\end{equation}
\end{theorem}
\begin{proof} 
By Proposition \ref{morphisms} it is enough to check when $x_pd_{pq}\Phi(v)=0$ for all $p,q\in [5]$.
For notational convenience we let $Q=(p,q)$ and $\{a,b,c,p,q\}=[5]$.
We have:
\begin{align*}
x_pd_Q\Phi(v)&=x_pd_Q\sum_{I\in{\mathcal I}_1/B_1}\omega_I\otimes \theta_I(v)=x_pd_Q\sum_{I\in{\mathcal I}_1/B_1}d_I\otimes \theta_I(v)\\
& =\sum_{I\in{\mathcal I}_1/B_1}\varepsilon_{Q,I}x_p\partial_{t_{Q,I}}.(\theta_I(v))=
\varepsilon_{pqabc}\sum_{(\alpha,\beta,\gamma)\in C(a,b,c)}
x_p\partial_\gamma. (\theta_{\alpha\beta}(v)).
\end{align*}
\end{proof}
\begin{remark} We point out that Equation (\ref{equazione1}) satisfies the hypotheses of Proposition \ref{lemdual} since in this case
\[x_p\partial_{\gamma}.(\theta_{\alpha\beta}(v))=\theta_{\alpha\beta}(x_p\partial_{\gamma}.v)\]
hence we can write
\[x_p\partial_{\gamma}.(\theta_{\alpha\beta}(v))=\frac{1}{2}\big(x_p\partial_{\gamma}.(\theta_{\alpha\beta}(v))+\theta_{\alpha\beta}(x_p\partial_{\gamma}.v)\big).\]
Conjecture \ref{conjdual} then holds in degree one. This will be also confirmed by Theorem \ref{gradouno}. 
\end{remark}

\begin{proposition}\label{pesigrado1} Let $\varphi: M(\lambda)\rightarrow M(\mu)$ be a morphism of Verma modules of degree one and let
$\theta_{hk}$ have the leading weight of $\varphi$. Then if $i<j$ are distinct from $h,k$ we have
\[
\mu_{ij}=-\chi_{i<h<j}-\chi_{i<k<j}.
\]
\end{proposition}
\begin{proof}
Consider Equation (\ref{equazione1}) with $p=j$, $c=i$, $a=h$, $b=k$ and $v=s$ a highest weight vector in $F(\lambda)$: 
\[
x_j\partial _i.(\theta_{hk}(s))+x_j\partial_k.(\theta_{ih}(s))+x_j\partial_h.(\theta_{ki}(s))=0.
\]
Now we apply $x_i\partial_j$ to this equation. We have 
\begin{align*}
h_{ij}.(\theta_{hk}(s))-\chi_{i<k<j}\theta_{kh}(s)-\chi_{i<h<j}\theta_{kh}(s)=0
\end{align*}
and the result follows.
\end{proof}
\begin{theorem}\label{gradouno}
 Let $\varphi: M(\lambda)\rightarrow M(\mu)$ be a morphism of Verma modules of degree one. 
Then one of the following occurs:
\begin{itemize}
\item 
$\lambda=(m,n+1,0,0)$, $\mu=(m,n,0,0)$ for some $m,n\geq 0$ and, up to a scalar, $\varphi=\nabla_A$.
\item 
$\lambda=(m+1,0,0,n)$, $\mu=(m,0,0,n+1)$ for some $m,n\geq 0$ and, up to a scalar, $\varphi=\nabla_B$.
\item 
$\lambda=(0,0,m,n)$, $\mu=(0,0,m+1,n)$ for some $m,n\geq 0$  and, up to a scalar, $\varphi=\nabla_C$.
\end{itemize}
\end{theorem}
\begin{proof} Let $\theta_{hk}$ have the leading weight of $\varphi$.
By Proposition \ref{pesigrado1} we have that if $(h,k)\neq (1,2),(1,5),(4,5)$ we can find $i,j$ such that $\mu_{i,j}<0$, a contradiction.
Proposition \ref{pesigrado1} also provides
\begin{itemize}
\item $\mu_{3,5}=0$ if $(h,k)=(1,2)$;
\item $\mu_{2,4}=0$ if $(h,k)=(1,5)$;
\item $\mu_{1,3}=0$ if $(h,k)=(4,5)$,
\end{itemize}
and the rest follows using Lemma \ref{esempi} and Proposition \ref{leading} recalling that $\lambda(\theta_{hk})=-\lambda(d_{hk})$.
\end{proof}

\section{Morphisms of degree 2}\label{eight}
In this section we provide a complete classification of morphisms between Verma modules of degree 2. We will make use of the following preliminary result which holds in a much wider generality.
Here and in what follows we denote by $(p,q,a,b,c)$ any permutation of $[5]$ and we set $Q=(p,q)$.
\begin{lemma} \label{diquaedila} 
Suppose that $\Phi=\sum_{T,I}\partial_T\omega_I\otimes\theta^T_I$ defines a morphism of Verma modules $\varphi: M(V)\rightarrow M(W)$. Then for all $t_1\dots t_h\in [5]$,
$I_1,\ldots,I_k\in \mathcal I_1$ and $v\in V$ we have
\begin{align*}
 &\sum_{I,J_1,\ldots,J_r\in\mathcal{I}_1}\varepsilon_{Q,I}x_p\partial_{t_{Q,I}}d_{J_1}\cdots d_{J_r}\otimes \theta^{t_1,\ldots,t_h}_{I_1,\ldots,I_k,I,J_1,\ldots,J_r}(v)
 =2\sum_{\substack{(\alpha,\beta, \gamma)\in C(a,b,c) \\ H_1,\ldots,H_r\in\mathcal{I}_1}}\varepsilon_{pqabc}d_{H_1} \cdots d_{H_r}\otimes\\
& \Big(\theta^{t_1,\ldots,t_h}_{I_1,\ldots,I_k,\alpha\beta,H_1,\ldots,H_r}(x_p\partial_\gamma.v)+\sum_{s=1}^h \Delta_{p\rightarrow\gamma}^s\theta^{t_1,\ldots,t_h}_{I_1,\ldots,I_k,\alpha\beta,H_1,\ldots,H_r}(v)-\sum_{s=1}^kD_{\gamma\rightarrow p}^s\theta^{t_1,\ldots,t_h}_{I_1,\ldots,I_k,\alpha\beta,H_1,\ldots,H_r}(v)\Big)
\end{align*}
\end{lemma}
\begin{proof}
Using  the definitions of $D_{a\rightarrow b}^h$, of $\Delta_{a\rightarrow b}^h,$ of $\theta^{t_1,\ldots,t_h}_{I_1,\ldots,I_k}$ and of the action of $L_0$ on the latter elements, we have
 \begin{align*}
&\sum_{I,J_1,\ldots,J_r}\varepsilon_{Q,I}x_p\partial_{t_{Q,I}}d_{J_1}\cdots d_{J_r}\otimes \theta^{t_1,\ldots,t_h}_{I_1,\ldots,I_k,I,J_1,\ldots,J_r}(v)= -2\sum_{\substack{(\alpha,\beta, \gamma)
\in C(a,b,c) \\ H_1,\ldots,H_r\in\mathcal{I}_1}}\varepsilon_{pqabc}d_{H_1} \cdots d_{H_r}\otimes\\
 &\Big((x_p \partial_\gamma. \theta^{t_1,\ldots,t_h}_{I_1,\ldots,I_k,\alpha\beta,H_1,\ldots,H_r})(v)-\sum_{s=1}^h \Delta_{p\rightarrow\gamma}^s\theta^{t_1,\ldots,t_h}_{I_1,\ldots,I_k,\alpha\beta,H_1,\ldots,H_r}(v)+\sum_{s=1}^kD^s_{\gamma\rightarrow p}\theta^{t_1,\ldots,t_h}_{I_1,\ldots,I_k,\alpha\beta,H_1,\ldots,H_r}(v)\\
 &- x_p\partial_\gamma (\theta^{t_1,\ldots,t_h}_{I_1,\ldots,I_k,\alpha\beta,H_1,\ldots,H_r}(v))\Big)\\
\end{align*} 
from which the thesis follows.
\end{proof}

We are now ready to state the following characterization result.

\begin{theorem}\label{equationdeg2}
 Let $\varphi:M(V)\rightarrow M(W)$ be a linear map of degree 2 associated to 
 \[\Phi=\sum_{(I,J)\in \mathcal I_2/B_2}\omega_{I,J}\otimes \theta_{I,J}+\sum_{t=1}^5\partial_t\otimes \theta^t\]
 such that $x.\Phi=0$ for all $x\in L_0$. Then $\varphi$ is a morphism of Verma modules if and only if for all  $K\in \mathcal I_1$ and all $v\in V$ we have
 \[
  -\chi_{(K\in B_1 Q)}\theta^p(v)+\frac{1}{2}\varepsilon_{pqabc}\sum_{(\alpha \beta \gamma) \in C(a,b,c)}\Big(-\big((x_p\partial_\gamma).\theta_{\alpha\beta , K}\big)(v)+2x_p\partial_\gamma .(\theta_{\alpha \beta,K}(v))\Big)=0
 \]
 \end{theorem} 
 \begin{proof}
By Proposition \ref{morphisms} we have that $\varphi$ is a morphism of Verma modules if and only if
\[
 x_pd_Q \Big(\sum_{(I,J)\in \mathcal I_2/B_2}\omega_{I,J}\otimes \theta_{I,J}(v)+\sum_{t}\partial_t\otimes \theta^t(v)\Big)=0
\]
for all $v\in V$. It is convenient for us to consider the first sum running over all $(I,J)\in \mathcal I_2$ and so we have
 
\begin{align}\label{aryu}
 \nonumber x_pd_{Q}\Big(\frac{1}{8}\sum_{(I,J)\in\mathcal{I}_2}&\omega_{I,J}\otimes \theta_{I,J}(v)+\sum_t \partial _t\otimes \theta^t(v) \Big)\\
&=x_pd_{Q}\Big(\frac{1}{8}\sum_{I,J}(d_Id_J-\frac{1}{2}\varepsilon_{I,J}\partial_{t_{I,J}})\otimes \theta_{I,J}(v)+\sum_t \partial _t\otimes \theta^t(v) \Big).
 \end{align}
 We split Equation \eqref{aryu} into three parts:

In the first part of Equation \eqref{aryu} we have, using Lemma \ref{diquaedila},
\begin{align*}
 x_pd_{Q}\sum_{I,J}d_Id_J\otimes \theta_{I,J}(v)&=\sum_{I,J}\Big(\varepsilon_{Q,I}(x_p \partial _{t_{Q,I}})d_J-\varepsilon_{Q,J}d_I (x_p \partial_{t_{Q,J}})\Big)\otimes \theta_{I,J}(v)\\
 &=2\sum_{H}d_H \otimes \varepsilon_{pqabc} \sum_{\alpha\beta\gamma}\theta_{\alpha\beta,H}(x_p\partial_\gamma.v)\\
 &-2\sum_{I}d_I\otimes \varepsilon_{pqabc}\sum_{\alpha \beta\gamma}(\theta_{I,\alpha\beta}(x_p\partial_\gamma.v)-D^1_{\gamma\rightarrow p}\theta_{I,\alpha\beta}(v))\\
 &=4\sum_Hd_H\otimes \varepsilon_{pqabc} \sum_{\alpha \beta\gamma}\big(\theta_{\alpha\beta,H}(x_p\partial_\gamma.v)+\frac{1}{2}(x_p\partial_\gamma.\theta_{\alpha\beta,H})(v)\big)\\
 &=4\sum_Hd_H\otimes \varepsilon_{pqabc} \sum_{\alpha \beta\gamma}\big(x_p\partial_\gamma.(\theta_{\alpha\beta,H}(v))-\frac{1}{2}(x_p\partial_\gamma.\theta_{\alpha\beta,H})(v)\big)
\end{align*}
where the sums run over $I,J,H\in\mathcal{I}_1$ and $(\alpha, \beta,\gamma)\in C(a,b,c)$.

In the second part of Equation \eqref{aryu} we have 
\begin{align*}
\sum_{I,J}\frac{1}{2}\varepsilon_{I,J}\partial_{t_{I,J}}\otimes \theta_{I,J}(v)=0
\end{align*}
since the term indexed by $(I,J)$ cancels the term indexed by $(J,I)$.

In the third part of Equation \eqref{aryu} we have:
\begin{align*}
\sum_t x_pd_Q \partial _t\otimes \theta^t(v)=- d_Q\otimes \theta^p(v).
\end{align*}
Putting the three parts together Equation \eqref{aryu} becomes
\begin{align*}
& x_pd_Q\Big(\frac{1}{8}\sum_{I,J\in\mathcal{I}_1}\omega_{I,J}\otimes \theta_{I,J}(v)+\sum_t \partial _t\otimes \theta^t(v) \Big)\\
 &=\sum_{K\in \mathcal I_1/B_1}d_K\otimes \Big(-\chi_{(K\in B_1Q)}\theta^p(v)+\varepsilon_{pqabc}\sum_{(\alpha \beta \gamma) \in C(a,b,c)}-\frac{1}{2}\big(x_p\partial_\gamma.\theta_{\alpha\beta , K}\big)(v)+x_p\partial_\gamma .(\theta_{\alpha \beta,K}(v))\Big)
\end{align*}
and the result follows.
\end{proof}

We deduce that Conjecture \ref{conjdual} holds for morphisms of degree 2 and in particular we have the following duality result for degree 2 morphisms.
\begin{corollary}\label{degree2dual}
 Let $\varphi:M(V)\rightarrow M(W)$ be a morphism of Verma modules of degree 2 associated to 
 \[\Phi=\sum_{(I,J)\in \mathcal I_2/B_2}\omega_{I,J}\otimes \theta_{I,J}+\sum_t \partial_t \otimes \theta^t.\]
 Then the linear map $\psi:M(W^*)\rightarrow M(V^*)$ associated to 
 \[
  \Psi=\sum_{(I,J)\in \mathcal I_2/B_2}\omega_{I,J}\otimes \theta^*_{I,J}+\sum_t \partial_t \otimes (-\theta^t)^*
 \]
 is also a morphism of Verma modules.

\end{corollary}
\begin{proof}
 This is an immediate consequence of Remark \ref{lemdual2} and Theorem \ref{equationdeg2}.
\end{proof}

\begin{corollary}\label{equationdegree2highest}
 Let $\varphi:M(\lambda)\rightarrow M(\mu)$ be a morphism of Verma modules and $s\in F(\lambda)$ a highest weight vector. Then for all  $K\in \mathcal I_1$ we have
 \[
  2 \chi_{K\in B_1Q}\varepsilon_{pqabc} \theta^p(s)+ \sum_{(\alpha\beta\gamma)\in C(abc)}\Big((-1)^{\chi_{p>\gamma}}(x_p\partial_{\gamma}.\theta_{\alpha\beta,K})(s)+2\chi_{p>\gamma}x_p\partial_\gamma. (\theta_{\alpha\beta,K}(s))\Big)=0
 \]
 
\end{corollary}
\begin{proof}
 This result immediately follows from Theorem \ref{equationdeg2} by observing that if $p<\gamma$ then $x_p\partial_\gamma.s=0$.
\end{proof}

In the following results we fix a morphism $\varphi:M(\lambda)\rightarrow M(\mu)$ of Verma modules of degree 2 associated to $\Phi=\sum \omega_{I,J}\otimes \theta_{I,J}+\sum \partial_t\otimes \theta^t$ and we exploit Corollary \ref{equationdegree2highest} to obtain some constraints on the weights $\lambda$ and $\mu$. The next result is analogous to Proposition \ref{pesigrado1}.
\begin{proposition}\label{proppesi1}Let $h,k,l,m \in [5]$ be such that $\theta_{hk,lm}$ has the leading weight of $\varphi$. Let $1\leq i<j\leq 5$ be such that $j\neq h,k,l,m $ and $i\neq h,k$. Then
\[
\mu_{ij}=-\chi_{i<h<j}-\chi_{i<k <j}.
\]
\end{proposition}
\begin{proof}By Corollary \ref{equationdegree2highest} used with $a=i$,  $b=h$, $c=k$, $p=j$ and $K=(l,m)$, observing that $x_j\partial_\gamma.\theta_{\alpha\beta,K}=0$ for all $(\alpha,\beta,\gamma)\in C(i,h,k)$, we obtain the following relation
\[
x_j\partial_i.(\theta_{hk,lm}(s))+\chi_{h<j} x_j\partial_h .(\theta_{ k i ,lm }(s))+\chi_{k<j}x_j\partial_k.( \theta_{i h, kl }(s))=0.
\]
Applying $x_i\partial_j$ to this equation we have
\begin{align*} 
h_{ij}.(\theta_{hk,lm}(s))&+\chi_{h<j}\big(x_i\partial_h.(\theta_{k i,lm }(s))-x_j\partial_h.(\theta_{k j,lm}(s))\big)\\
&+\chi_{k <j}\big(x_i\partial_k.(\theta_{ih,lm }(s))-x_j\partial_k.(\theta_{jh,lm }(s))\big)=0
\end{align*}
Since $\theta_{hk,lm}$ has the leading weight of $\varphi$, if $h<j$ we necessarily have $\theta_{k j,lm }(s)=0$, by Corollary \ref{dominancetheta}.  Similarly, if $k<j$, we have $\theta_{jh,lm }(s)=0$. Therefore the previous equation becomes 
\[ 
h_{ij}.(\theta_{hk,lm }(s))+\chi_{h<j}x_i\partial_h.(\theta_{k i,lm }(s))
+\chi_{k <j}x_i\partial_k.( \theta_{ih,lm }(s))=0
\]
Again, if $i>h$, we have $\theta_{k i,lm }(s)=0$ and otherwise we have $x_i\partial_h.(\theta_{k i,lm }(s))=-\theta_{kh,lm}(s)$ and similarly for the other term, and so we have

\[
h_{ij}.(\theta_{hk,lm}(s))-\chi_{h<j}\chi_{i<h}\theta_{kh,lm}(s)
-\chi_{k <j}\chi_{i<k }\theta_{kh,lm}(s)=0
\]
i.e.,
\[
h_{ij}.(\theta_{hk,lm}(s))=-(\chi_{i<h<j}+\chi_{i<k <j})\theta_{hk,lm}(s).
\]
\end{proof}
\begin{proposition}\label{proppesi3}
Let $i,h,k,l,m\in [5]$, with $i,h,k,m$ distinct and $i<m$, be such that $\theta_{hk,lm}$ has the leading weight of $\varphi$. Then
\begin{align*}
&h_{im}.( \theta_{hk,lm}(s))=\\
&\Big(\frac{1}{2}-\chi_{i<h<m}-\chi_{i<k<m}\Big)\theta_{hk,lm}(s)-\varepsilon_{mlhki}\theta^{i}(s)-\frac{1}{2}\Big((-1)^{\chi_{h<m}}\theta_{hl,km}(s)+(-1)^{\chi_{k<m}}\theta_{hm,kl}\Big).
\end{align*}

\end{proposition}
\begin{proof}
We consider Corollary \ref{equationdegree2highest} with $a=h$, $b=k$, $c=i$, $p=m$ and $K=(l,m)$. We observe that \[\varepsilon_{pqabc}\chi_{K\in B_1Q}=\varepsilon_{mqhki}\chi_{l=q}=\varepsilon_{mlhki}\] and so we obtain
\begin{align*}
\varepsilon_{mlhki}\theta^m(s)
&+\frac{1}{2}\Big((-1)^{\chi_{h<m}}\theta_{k i ,hl}(s)
+(-1)^{\chi_{k<m}}\theta_{i h,kl}(s)
-\theta_{hk,il}(s)\Big)\\
&+\chi_{h<m}x_m \partial_h.( \theta_{k i ,lm}(s))
+\chi_{k<m}x_m \partial_k. (\theta_{i h ,lm}(s))
+x_m \partial_i. (\theta_{hk,lm}(s))=0
\end{align*}
We apply $x_i\partial_m$ to this equation and we obtain
\begin{align*}
\varepsilon_{mlhki}\theta^i(s)
&-\frac{1}{2}\Big((-1)^{\chi_{h<m }}\theta_{km,hl}(s)
+(-1)^{\chi_{k<m}}\theta_{mh,kl}(s)
+\theta_{hk,ml}(s)\Big)\\
&-\chi_{i<h <m }\theta_{kh,lm}(s)
-\chi_{i<k <m } \theta_{kh ,lm}(s)
+h_{im}.(\theta_{hk,lm}(s))=0
\end{align*}
and the result follows.
\end{proof}

\begin{proposition}\label{proppesi2}
Let $h,k,m,i\in [5]$ be distinct, $i<m$, be such that $\theta_{hk,hm}$ has the leading weight of $\varphi:M(\lambda)\rightarrow M(\mu)$.  Then
\[
\mu_{i,m}=\chi_{k<m}-\chi_{i<h<m}-\chi_{i<k<m}
\]
and
\[
\lambda_{i,m}=\chi_{k<m}-\chi_{i<h<m}-\chi_{i<k<m}-1.
\]
\end{proposition}
\begin{proof}
We use Proposition \ref{proppesi3} with $l=h$ and deduce
\begin{align*}
h_{im}.(\theta_{hk,hm}(s))&=\Big(\frac{1}{2}-\chi_{i<h<m}-\chi_{i<k<m}\Big)\theta_{hk,hm}(s)-\frac{1}{2}(-1)^{\chi_{k<m}}\theta_{hm,kh}\\
&=\Big(\frac{1}{2}-\frac{1}{2}(-1)^{\chi_{k<m}}-\chi_{i<h<m}-\chi_{i<k<m}\Big)\theta_{hk,hm}(s)\\
&=(\chi_{k<m}-\chi_{i<h<m}-\chi_{i<k<m}\Big)\theta_{hk,hm}(s).
\end{align*}
and the first part of the statement follows. The second part is an easy consequence since
\[
\lambda_{i,m}(\theta_{hk,hm})=1.
\]
\end{proof}

\begin{theorem}\label{teorema2} Let $\varphi:M(\lambda)\rightarrow M(\mu)$ be a morphism of degree 2.
Then one of the following occurs:
\begin{enumerate}
 \item 
$\lambda=(1,0,0,n)$, $\mu=(0,0,1,n+1)$ for some $n\geq 0$ and, up to a scalar, $\varphi=\nabla_C\nabla_B$;
 \item  
$\lambda=(n+1,1,0,0)$, $\mu=(n,0,0,1)$ for some $n\geq 0$ and, up to a scalar, $\varphi=\nabla_B\nabla_A$;
 \item 
$\lambda=(0,1,0,0)$, $\mu=(0,0,1,0)$, and, up to a scalar, $\varphi=\nabla_C\nabla_A$.
\end{enumerate}

\end{theorem}
\begin{proof}
We first make the following observation that will allow us to simplify several arguments. If $\nu\in \Lambda$ is any weight,  by Corollary \ref{degree2dual}, if the statement holds for all morphisms of leading weight $\nu$ then it holds also for all morphisms of leading weight $-\nu^*$. 

We let $s$ be a highest weight vector of $F(\lambda)$ and we suppose that $\theta_{hk,lm}$ has the leading weight of $\varphi$. Let us first assume $|\{h,k,l,m\}|=3$ i.e., without loss of generality,  $h=l$. 

By Corollary \ref{equationdegree2highest} with $K=(p,a)$
we have:
\begin{align}
\nonumber-((-1)^{\chi_{b<p}}&+(-1)^{\chi_{c<p}})\theta_{ab,ca}(s)\\+
&2\chi_{a<p}x_p\partial_a.(\theta_{bc,pa}(s))
+2\chi_{b<p}x_p\partial_b.(\theta_{ca,pa}(s))+2\chi_{c<p}x_p\partial_c.(\theta_{ab,pa}(s))=0.
\label{equation2m=a}
\end{align}
Using this equation with $a=h$, $b=k$, $c=m$, since $\theta_{hk,hm}$ has the leading weight of $\varphi$, we immediately obtain 
\[
((-1)^{\chi_{k<p}}+(-1)^{\chi_{m<p}})\theta_{hk,hm}(s)=0.
\]

In particular, if we can choose $p$ such that $p>k,m$ or $p<k,m$ we have $\theta_{hk,hm}(s)=0$, a contradiction. 
So we reduce to study the following cases: (a) $k=1, m=5$; (b) $k=2, m=5, h=1$; (c) $k=1, m=4, h=5.$
\begin{enumerate}
\item [(a)] By duality, since $\lambda(\theta_{21,25})=-(\lambda(\theta_{41,45}))^*$, it is enough to consider only the cases $h=2,3$; we have, by Proposition \ref{proppesi1},
\[
\mu_{14}=-\chi_{1<h<4}-\chi_{1<5<4}=-1,
\]
a contradiction. 
\item [(b)] In this case we have, by Proposition \ref{proppesi1} 
\[
\mu_{23}=-\chi_{2<1<3}-\chi_{2<5<3}=0
\]
and by Proposition \ref{proppesi2} we have
\[
\mu_{35}=\chi_{2<5}-\chi_{3<1<5}-\chi_{3<2<5}=1.
\]
Since the leading weight of $\varphi$ is $\lambda(\theta_{12,15})=(-1,-1,0,1)$ we conclude that $\mu=(n,0,0,1)$ for some $n\geq 0$ and so $\lambda=(n+1,1,0,0)$.
The leading term of the singular vector $\varphi(1\otimes s)$ is $\omega_{12,15}\otimes \theta_{12,15}(s)=d_{12}d_{15}\otimes \theta_{12,15}(s)$ hence, up to a scalar,
 $\varphi=\nabla_B \nabla_A$ by Proposition \ref{leading}.
\item[(c)] Since $\lambda(\theta_{51,54})=-\lambda(\theta_{12,15})^*$ this follows from case (b) and we obtain in this case the morphism $\nabla_C\nabla_B$.
\end{enumerate}

This concludes the study of all possible $\theta_{hk,lm}$ having the leading weight of $\varphi$ with $h,k,l,m$ not distinct.

In order to deal with the case where $h,k,l,m$ are distinct we let $p$ be different from $h,k,l,m$.
If $p=4,5$ we apply Proposition \ref{proppesi1} with $i=1$ and $j=p$ and we get that $\mu_{1 p}<0$ hence $\theta_{hk,lm}$ does not have the leading
weight of $\varphi$. By Corollary \ref{equationdegree2highest} we also have $\theta^p(s)=0$ and so also $\theta^p$ can not have the leading weight of $\varphi$.

For $p=1$ we have $\lambda(\theta^1)=-\lambda(\theta^5)^*$ and if $p=2$ we have  $\lambda(\theta^2)=-\lambda(\theta^4)^*$ and so these cases follows from the previous discussion
by Corollary \ref{degree2dual}.

For $p=3$ Proposition \ref{proppesi1} with $i=1$, $j=3$ shows that $\theta_{14,25}$ and $\theta_{15,24}$ cannot have the leading weight of $\varphi$,
i.e.\ $\theta_{14,25}(s)=\theta_{15,24}(s)=0$, and
that if $\theta_{12,45}$ has leading weight then $\mu_{1,3}=0$. Besides, by Corollary \ref{equationdegree2highest}, 
$\theta_{12,45}(s)=2\theta^3(s)$.
By Proposition \ref{proppesi3} we immediately get 
\[
 h_{35}.(\theta_{12,45}(s))=\theta_{12,45}(s)
\]
and so $\mu_{3,5}=1$.
Since  the leading weight is $\lambda(\theta_{12,45})=(0,-1,1,0)$ we conclude that $\mu=(0,0,1,0)$ and so $\lambda=(0,1,0,0)$.
The leading term of $\varphi(1\otimes s)$ is 
$$\omega_{12,45}\otimes \theta_{12,45}(s)+\partial_3\otimes\theta^3(s)=d_{12}d_{45}\otimes \theta^3(s)$$ 
hence, up to a scalar, $\varphi=\nabla_{C}\nabla_A$ by Proposition \ref{leading}.


\end{proof}
\section{Morphisms of degree 3}\label{nine}
This section is dedicated to the study of morphisms of Verma modules of degree three. We consider a linear map  $\varphi: M(\lambda)\rightarrow M(\mu)$
 of degree three associated to \[\Phi=\sum_{I\in{\mathcal I}_3/B_3}\omega_I\otimes \theta_I+\sum_{t\in[5], I\in{\mathcal I}_1/B_1}\partial_t\omega_I\otimes \theta^t_I.\] 
As in the case of morphisms of degree one and two, our goal is to establish necessary and sufficient conditions to ensure that $\varphi$ is a morphism of Verma modules.
\begin{lemma}\label{fybj} If $x.\Phi=0$ for every $x\in L_0$, then
the following relation holds for every $v\in F(\lambda)$:
\[
\sum_{I\in \mathcal I_3}\omega_I\otimes \theta_I(v)=\sum_{I\in \mathcal I_3}d_I\otimes \theta_I(v).
\]
\end{lemma}
\begin{proof}
Indeed we have
\begin{align*}
 \sum_{I\in \mathcal I_3}\omega_I\otimes \theta_I(v)&=\sum_{I\in \mathcal I_3}d_I\otimes \theta_I(v)\\ 
 & +\sum_{I_1,I_2,I_3}(-\frac{1}{2}
\varepsilon_{I_1,I_2}\partial_{t_{I_1,I_2}}d_{I_3}+\frac{1}{2}
\varepsilon_{I_1,I_3}\partial_{t_{I_1,I_3}}d_{I_2}- \frac{1}{2}
\varepsilon_{I_2,I_3}\partial_{t_{I_2,I_3}}d_{I_1})\otimes \theta_{I_1,I_2,I_3}(v) \end{align*}
and the last sum vanishes since the coefficients of $\theta_{I_1,I_2,I_3}(v)$ and  $\theta_{I_3,I_2,I_1}(v)$ coincide.
\end{proof}

\begin{theorem}\label{grado3} Let us assume that $x.\Phi=0$ for every $x\in L_0$. Then $\varphi$ is a morphism of Verma modules if and only if for every $H,L\in \mathcal{I}_1$, every permutation $(p,q,a,b,c)$ of $[5]$ and every $v\in F(\lambda)$, the following equations hold:
\begin{align}
\label{deg3n1}&  \chi_{L\in B_1Q}\theta^p_{H}(v)+\frac{1}{2}\varepsilon_{pqabc}\sum_{(\alpha,\beta,\gamma)\in C(a,b,c)}
\big(-(x_p\partial_{\gamma}.\theta_{\alpha\beta,H,L})(v)+2 x_p\partial_\gamma. (\theta_{\alpha \beta, H,L}(v))\big)=0\\
\label{deg3n2}&\frac{1}{4}\theta_{ab,bc,cq}(v)+\frac{1}{4}\theta_{ac,cb,bq}(v)+\frac{1}{2}\varepsilon_{pqabc}\sum_{(\alpha,\beta,\gamma)\in C(a,b,c)}\big(-(x_p\partial_\gamma.\theta^a_{\alpha\beta})(v)+2x_p\partial_\gamma.(\theta^a_{\alpha\beta}(v))\big)=0\\
\label{deg3n3}& \sum_{(\alpha, \beta, \gamma)\in C(a,b,c)}x_p\partial_{\gamma}.(\theta^p_{\alpha \beta}(v))=0\\
\label{deg3n4}& \varepsilon_{pqabc}\sum_{(\alpha, \beta, \gamma)\in C(a,b,c)}x_p\partial_{\gamma}.(\theta^q_{\alpha \beta}(v))-\frac{1}{2}\theta_{ab,bc,ca}(v)=0.
\end{align}
\end{theorem}
\begin{proof} By Proposition \ref{morphisms} we need to compute $x_pd_Q\Phi(v)$ for $v\in F(\lambda)$. We compute the different summands separately.
Using Lemma \ref{fybj} and Lemma \ref{diquaedila} we have
\begin{align*} x_pd_Q &\sum_{(I,J,K)\in\mathcal{I}_3/B_3}\omega_{I,J,K}\otimes \theta_{I,J,K}(v)=
\frac{1}{48}\sum_{I,J,K\in\mathcal{I}_3}\omega_{I,J,K}\otimes \theta_{I,J,K}(v)\\
 &=\frac{1}{48} x_pd_Q \sum_{I,J,K}d_I d_J d_K\otimes \theta_{I,J,K}(v)\\
 &=\frac{1}{48}\sum_{I,J,K}(\varepsilon_{Q,I}x_p\partial _{t_{Q,I}}d_Jd_K-d_I \varepsilon_{Q,J}\,x_p\partial_{t_{Q,J}}d_K+d_Id_J\varepsilon_{Q,K}x_p\partial_{t_{Q,K}})\otimes {\theta_{I,J,K}(v)}\\
 &=\frac{1}{48} \sum_{H,L}d_Hd_L\otimes 2\varepsilon_{pqabc}\sum_{\alpha\beta\gamma}\Big(D^2_{\gamma \rightarrow p}\theta_{\alpha\beta,H,L}(v)+
2D^3_{\gamma \rightarrow p}\theta_{\alpha \beta, H,L}(v)+3x_p\partial_\gamma.(\theta_{\alpha \beta, H,L}(v))\Big),
\end{align*}
where the sums run over $I,J,K\in \mathcal I_1$ and $(\alpha,\beta,\gamma)\in C(a,b,c)$. 

Recalling that $d_Hd_L=\omega_{H,L}+\frac{1}{2}\varepsilon_{H,L}\partial_{t_{H,L}}$  we have:
\begin{align*}
x_pd_Q &\sum_{(I,J,K)\in \mathcal I_3/B_3}\omega_{I,J,K}\otimes \theta_{I,J,K}(v)\\
&= \frac{1}{48}\sum_{H,L}\omega_{H,L}\otimes 2\varepsilon_{pqabc}\sum_{\alpha\beta\gamma }\Big(D^2_{\gamma\rightarrow p}\theta_{\alpha\beta,H,L}(v)+
2D^3_{\gamma \rightarrow p}\theta_{\alpha \beta, H,L}(v)+3x_p\partial_\gamma.( \theta_{\alpha \beta, H,L}(v))\Big)\\ 
&\,\,\,+\frac{1}{48}\sum_{H,L}\partial_{t_{H,L}}\otimes \varepsilon_{H,L}\varepsilon_{pqabc}\sum_{\alpha\beta\gamma }\Big(D^2_{\gamma\rightarrow p}\theta_{\alpha\beta,H,L}(v)
+2D^3_{\gamma\rightarrow p}\theta_{\alpha \beta, H,L}(v)+3x_p\partial_\gamma.(\theta_{\alpha \beta, H,L}(v))\Big)\\
&=\frac{1}{48}\sum_{(H,L)\in\mathcal{I}_2/B_2}\omega_{H,L}\otimes 2\varepsilon_{pqabc}\sum_{\alpha\beta\gamma }\Big(12D^2_{\gamma\rightarrow p}\theta_{\alpha\beta,H,L}(v)+
12D^3_{\gamma\rightarrow p}\theta_{\alpha \beta, H,L}(v)+24x_p\partial_\gamma.( \theta_{\alpha \beta, H,L}(v))\Big)\\ 
&\,\,\,+\frac{1}{48}\sum_{(H,L)\in\mathcal{I}_2/B_2}\partial_{t_{H,L}}\otimes \varepsilon_{H,L}\varepsilon_{pqabc}\sum_{\alpha\beta\gamma }\Big(-4D^2_{\gamma\rightarrow p}\theta_{\alpha\beta,H,L}(v)+4 D^3_{\gamma\rightarrow p}\theta_{\alpha \beta, H,L}(v)\Big)\\
&=\sum_{(H,L)\in\mathcal{I}_2/B_2}\omega_{H,L}\otimes \frac{1}{2}\varepsilon_{pqabc}\sum_{\alpha\beta\gamma }
\Big(-(x_p\partial_{\gamma}.\theta_{\alpha\beta,H,L})(v)+
2x_p\partial_\gamma.( \theta_{\alpha \beta, H,L}(v))\Big)\\
&\,\,\,+\partial_q\otimes- \frac{1}{2} \theta_{ab,bc,ca}(v) +\sum_{\alpha\beta\gamma}\partial_\alpha\otimes \frac{1}{4} (\theta_{\alpha\beta,\beta\gamma,\gamma q}(v)+\theta_{\alpha\gamma,\gamma \beta,\beta q}(v)).
\end{align*}

We also need the following computation
\begin{align*}
 x_pd_Q \sum_{t\in [5]}\sum_{I\in \mathcal I_1/B_1} &\partial_t\omega_I\otimes \theta^t_I(v)=-\frac{1}{2}\sum_{I\in \mathcal I_1}d_Qd_I\otimes \theta^p_I(v)+\frac{1}{2}\sum_{t}\partial_t x_pd_Q\sum_I d_I\otimes \theta^t_I(v)\\
&=-\frac{1}{2}\sum_Id_Qd_I\otimes \theta^p_I(v)+\sum_t\partial_t\otimes \varepsilon_{pqabc}\sum_{\alpha\beta\gamma}x_p\partial_\gamma.( \theta^t_{\alpha\beta}(v))\\
&=\frac{1}{2}\sum_I (\omega_{I,Q}-\frac{1}{2}\varepsilon_{Q,I}\partial_{t_{Q,I}})\otimes \theta^p_I(v)+\sum_t\partial_t\otimes \varepsilon_{pqabc}\sum_{\alpha\beta\gamma}x_p\partial_\gamma.( \theta^t_{\alpha\beta}(v))\\
&= \sum_{I\in \mathcal I_1/B_1} (\omega_{I,Q}\otimes \theta^p_I(v)-\partial_{t_{Q,I}}\otimes\frac{1}{2}\varepsilon_{Q,I}\theta^p_I(v))+\sum_t\partial_t\otimes \varepsilon_{pqabc}\sum_{\alpha\beta\gamma}x_p\partial_\gamma.( \theta^t_{\alpha\beta}(v)).
\end{align*}

Now we can use these two relations and compute
\begin{align*}
x_p&d_Q\Phi(v)=x_pd_Q\Big(\sum_{(I,J,K)\in \mathcal I_3/B_3}\omega_{I,J,K}\otimes \theta_{I,J,K}(v)+\sum_{t\in[5]}\sum_{I\in \mathcal I_1/B_1}\partial_t\omega_I\otimes \theta^t_{I}(v)\Big)\\
&=\sum_{(H,L)\in\mathcal{I}_2/B_2}\omega_{H,L}\otimes \frac{1}{2}\varepsilon_{pqabc}\sum_{\alpha\beta\gamma }
\Big(-(x_p\partial_{\gamma}.\theta_{\alpha\beta,H,L})(v)+
2x_p\partial_\gamma.( \theta_{\alpha \beta, H,L}(v))\Big)\\
&\,\,\,+\partial_q\otimes- \frac{1}{2} \theta_{ab,bc,ca}(v) +\sum_{\alpha\beta\gamma}\partial_\alpha\otimes \frac{1}{4} (\theta_{\alpha\beta,\beta\gamma,\gamma q}(v)+\theta_{\alpha\gamma,\gamma \beta,\beta q}(v))\\
&+\sum_{I\in \mathcal I_1/B_1} (\omega_{I,Q}\otimes \theta^p_I(v)-\partial_{t_{Q,I}}\otimes\frac{1}{2}\varepsilon_{Q,I}\theta^p_I(v))+\sum_t\partial_t\otimes \varepsilon_{pqabc}\sum_{\alpha\beta\gamma}x_p\partial_\gamma.( \theta^t_{\alpha\beta}(v))\\
&=\sum_{(H,L)\in\mathcal{I}_2/B_2}\omega_{H,L}\otimes \big(\chi_{L\in B_1Q}\theta^p_{H}(v)+\varepsilon_{pqabc}\sum_{\alpha\beta\gamma}
\big(-\frac{1}{2}(x_p\partial_{\gamma}.\theta_{\alpha\beta,H,L})(v)+
x_p\partial_\gamma.(\theta_{\alpha \beta, H,L}(v))\big)\\
&\hspace{5mm}+\partial_p\otimes\varepsilon_{pqabc}\sum_{\alpha\beta\gamma}x_p\partial_\gamma.( \theta^p_{\alpha\beta}(v))+\partial_q\otimes \big(\varepsilon_{pqabc}\sum_{\alpha\beta\gamma}x_p\partial_\gamma.( \theta^q_{\alpha\beta}(v))-\frac{1}{2}\theta_{ab,bc,ca}(v)\big)\\
&\hspace{5mm}+\sum_{\alpha\beta\gamma}\partial_\alpha \otimes\Big(\frac{1}{4}\theta_{\alpha\beta,\beta\gamma,\gamma q}(v)+\frac{1}{4}\theta_{\alpha\gamma,\gamma\beta,\beta q}(v)+\varepsilon_{pqabc}\big(-\frac{1}{2}\theta^p_{\beta \gamma}(v)+x_p\partial_c.( \theta^\alpha_{ab}(v))\\&\hspace{30mm}+x_p \partial_b .(\theta^\alpha_{ca}(v))+x_p\partial_a .(\theta^\alpha_{bc}(v))\big)\Big).
\end{align*}
This completes the proof of Equations \eqref{deg3n1}, \eqref{deg3n3} and \eqref{deg3n4}. In order to deduce Equation \eqref{deg3n2} we consider the coefficient of $\partial_a$ in the previous equation (the coefficients of $\partial_b$ and $\partial_c$ provide equivalent conditions) and we have
\begin{align*}
 \frac{1}{4}&\theta_{ab,bc,cq}(v)+\frac{1}{4}\theta_{ac,cb,bq}(v)+\varepsilon_{pqabc}\big(-\frac{1}{2}\theta^p_{bc}(v)+x_p\partial_c .(\theta^a_{ab}(v))+x_p \partial_b.( \theta^a_{ca}(v))+x_p\partial_a. (\theta^a_{bc}(v))\big)\\
 &=\frac{1}{4}\theta_{ab,bc,c q}(v)+\frac{1}{4}\theta_{ac,cb,bq}(v)+\varepsilon_{pqabc}\Big(-\frac{1}{2}\big((x_p\partial_a.\theta^a_{bc})(v)+(x_p\partial_b. \theta^a_{ca})(v)+(x_p\partial_c. \theta^a_{ab})(v)\big)\\ &\hspace{5mm}+x_p\partial_c .(\theta^a_{ab}(v))+x_p \partial_b.( \theta^a_{ca}(v))+x_p\partial_a. (\theta^a_{bc}(v))\Big)\\
 &=\frac{1}{4}\theta_{ab,bc,cq}(v)+\frac{1}{4}\theta_{ac,cb,bq}(v)+\frac{1}{2}\varepsilon_{pqabc}\sum_{\alpha\beta\gamma}\big(-(x_p\partial_\gamma.\theta^a_{\alpha\beta})(v)+2x_p\partial_\gamma.(\theta^a_{\alpha\beta}(v))\big).
 \end{align*}
\end{proof}
\begin{corollary}\label{degree3dual}
 Let $\varphi:M(\lambda)\rightarrow M(\mu)$ be a morphism of Verma modules of degree 3 associated to 
 \[\Phi=\sum_{I\in \mathcal I_3/B_3}\omega_I\otimes \theta_I+\sum_{t\in [5]}\sum_{I\in \mathcal I_1/B_1} \partial_t\omega_I \otimes \theta^t_I.\]
 Then the linear map $\psi:M(\mu^*)\rightarrow M(\lambda^*)$ associated to 
 \[
  \Psi=\sum_{I\in \mathcal I_3/B_3}\omega_I\otimes \theta_I^*+\sum_{t\in [5]}\sum_{I\in \mathcal I_1/B_1} \partial_t\omega_I \otimes (-\theta^t_I)^*
 \]
 is also a morphism of Verma modules.
\end{corollary}
\begin{proof}
 This is an immediate consequence of Remark \ref{lemdual2} and Theorem \ref{grado3}.
\end{proof}

If we consider Equation \eqref{deg3n1} on a highest weight vector $s\in F(\lambda)$ (and we multiply it by $2\varepsilon_{pqabc}$) we obtain the following equation:

\begin{equation}\label{maxdeg3}
 2\varepsilon_{pqabc}\chi_{L\in B_1Q}\theta^p_{H}(s)+\sum_{\alpha\beta\gamma}
\big((-1)^{\chi_{p>\gamma}}(x_p\partial_{\gamma}.\theta_{\alpha\beta,H,L})(s)+2\chi_{p>\gamma}x_p\partial_\gamma.(\theta_{\alpha\beta,H,L}(s))\big)=0.
\end{equation}

\begin{remark}\label{asdf}
 If  $x_p\partial_c.\theta_{ab,H,L}$ has the leading weight of $\varphi$ then $\chi_{p>\gamma}x_p\partial_\gamma.(\theta_{\alpha\beta,H,L}(s))=0$ for all $(\alpha,\beta,\gamma)\in C(a,b,c)$ and so we obtain the following
 \begin{equation}
  2\varepsilon_{pqabc}\chi_{L\in B_1Q}\theta^p_{H}(s)+\sum_{\alpha\beta\gamma}
(-1)^{\chi_{p>\gamma}}(x_p\partial_{\gamma}.\theta_{\alpha\beta,H,L})(s)=0.
 \end{equation}
\end{remark}
This equation has several immediate consequences.
\begin{lemma} \label{a,b,c,d}If $a,b,c,d\in [5]$ are distinct then $\theta_{ab,ac,ad}$ does not have the leading weight of $\varphi$.
\end{lemma}
\begin{proof}
 Without loss of generality we can assume that the fifth element $p$ is either bigger than both $b$ and $c$ or smaller than both $b$ and $c$. Otherwise we can rename $b,c,d$ accordingly. Remark \ref{asdf} applies with $H=(a,p)$, $q=d$ and $L=(a,d)$ so we have
 \[
  (-1)^{\chi_{p>c}}\theta_{ab,ac,ad}(s)+(-1)^{\chi_{p>b}}\theta_{ab,ac,ad}(s)=0.
 \]
\end{proof}
\begin{lemma}\label{a,b,c}
If $a,b,c\in [5]$ are distinct then $\theta_{ab,bc,ca}$ does not have the leading weight of $\varphi$.
 \end{lemma}
\begin{proof}
 Without loss of generality we can choose $p$ such that $p$ is either bigger than both $a$ and $c$ or smaller than both $a$ and $c$. Remark \ref{asdf} applies with $H=(b,p)$ and $L=(c,a)$ so we have
 \[
  (-1)^{\chi_{p>c}}\theta_{ab,bc,ca}(s)+(-1)^{\chi_{p>a}}\theta_{ab,bc,ca}(s)=0.
 \]

\end{proof}
\begin{lemma}\label{12e45}
If $x,y,z,w\in[5]$ are distinct and $\theta_{xy,zw,xw}$ has the leading weight of $\varphi$, then  $\theta_{xy,zw,xw}=\theta_{12,45,kl}$ for some $k,l\in\{1,2,4,5\}$.
\end{lemma}
\begin{proof}
Let us first assume that $\{x,y,z,w\}\neq \{1,2,4,5\}$.
This assumption ensures that we can assume that the fifth element $p$ is either bigger or smaller than both $y$ and $w$ (otherwise exchange the roles of $x,z$ and $y,w$).
 Use Remark \ref{asdf} with $a=x$, $b=y$, $c=w$, $q=z$, $H=(x,p)$, $L=(z,w)$. Then we have:
  \[
  (-1)^{\chi_{p>y}}\theta_{xy,xw,zw}(s)+(-1)^{\chi_{p>w}}\theta_{xy,xw,zw}(s)=0.
 \]
Now let $\{x,y,z,w\}= \{1,2,4,5\}$. If either
$\{y,w\}=\{1,2\}$ or $\{y,w\}=\{4,5\}$ then we can use the same argument as above.

Now let $\{y,z\}=\{4,5\}$ so that $\theta_{1y,2z,12}$ has the leading weight of $\varphi$. Equation \eqref{maxdeg3} with $a=1$, $b=2$, $q=3$, $c=y$ $p=z$, $H=(2,p)$ and $L=(1,2)$ gives
$$x_z\partial_2.(\theta_{1y,2z,12}(s))=0$$
hence if we apply $x_2\partial_z$ we get $h_{2z}.(\theta_{1y,2z,12}(s))=0$ which implies in particular that
$\mu_{34}=0$. Since $\lambda_{34}(\theta_{1y,2z,12})=1$ this contradicts the dominance of $\lambda$.
The thesis follows.
\end{proof}
\begin{lemma} \label{1245}The elements  $\theta_{12,45,14}$, $\theta_{12,45,25}$ and $\theta_{12,45,24}$ do not have the leading weight of $\varphi$.
\end{lemma}
\begin{proof}
 Use Equation \eqref{maxdeg3} with $a=1$ $b=2$ $c=4$, $q=3$ and $p=5$, $H=(4,5)$ and $L=(1,2)$. We obtain
\begin{equation}\label{czwp}
\theta_{24,41,12}(s)+\theta_{41,42,12}(s)+2x_5\partial_1. (\theta_{24,45,12}(s))+2x_5\partial_2. (\theta_{41,45,12}(s))=0.
\end{equation}
Assume $\theta_{12,45,14}$ has the leading weight of $\varphi$. Then $\theta_{24,45,12}(s)=0$ and we apply $x_2\partial_5$ to Equation \eqref{czwp} to obtain
\[
-\theta_{54,41,12}(s)-\theta_{24,41,15}(s)-\theta_{41,45,12}(s)-\theta_{41,42,15}(s)+2h_{25}.(\theta_{41,45,12}(s))=0.
\]
But by Lemma \ref{12e45} we have $\theta_{24,41,15}(s)=0$ and so we have
\[
-2\theta_{41,45,12}(s)+2h_{25}.(\theta_{41,45,12}(s))=0.
\]
It follows that $\lambda_{25}(\theta_{41,45,12}(s))=1$ and so $\lambda_{34}(\theta_{41,45,12}(s))\leq 1$ and, since $\lambda_{34}(\theta_{41,45,12})=2$ this would imply $\lambda_{34}(s)\leq -1$, a contradiction.

By Corollary \ref{degree3dual} the element $\theta_{12,45,25}$ does not have the leading weight of $\varphi$ since $\lambda(\theta_{12,45,25})=-\lambda(\theta_{12,45,14})^*$.

Now we assume that $\theta_{12,45,24}$ has the leading weight of  $\varphi$. We apply $x_1\partial_5$ to Equation \eqref{czwp} to obtain
\[
-\theta_{24,45,12}(s)-\theta_{24,41,52}(s)-\theta_{45,42,12}-\theta_{41,42,52}+2h_{15}.(\theta_{24,45,12}(s))+2x_1\partial_2.(\theta_{41,45,12}(s))=0.
\]
Lemma \ref{12e45} ensures $\theta_{24,41,52}(s)=0$ and so we obtain
\[
-2\theta_{24,45,12}(s)+2h_{15}.(\theta_{24,45,12}(s)-2\theta_{42,45,12}(s)=0
\]
and we conclude
\[
h_{15}.(\theta_{24,45,12}(s))=0.
\]
We obtain a contradiction with  the same argument used in the other case.

\end{proof}
\begin{lemma}\label{morfismogrado3}
Assume that $\theta_{12,15,45}$ has the leading weight of $\varphi$. Then $\lambda=(1,1,0,0)$, $\mu=(0,0,1,1)$ and $\varphi=\nabla_C\nabla_B\nabla_A$ (up to a scalar).
\end{lemma}
\begin{proof}
Use Equation \eqref{maxdeg3} with $a=1$, $b=2$, $c=4$, $q=3$, $p=5$, $H=(1,5)$ and $L=(4,5)$. We obtain
\[
\theta_{12,14,45}(s)+\theta_{24,15,41}(s)+\theta_{41,12,45}(s)+\theta_{41,15,42}(s)+2x_5\partial_4.(\theta_{12,15,45}(s))=0
\]
since
$\theta_{24,14,45}(s)=\theta_{41,15,45}(s)=0$.
Applying $x_4\partial_5$ we get
\[
-\theta_{12,15,45}(s)-\theta_{25,15,41}(s)-\theta_{51,12,45}(s)-\theta_{41,15,52}(s)+2h_{45}.(\theta_{12,15,45}(s))=0.
\]
By Lemma \ref{12e45} we have $\theta_{25,15,41}(s)=0$ and so we obtain
\[
-2\theta_{12,15,45}(s)+2h_{45}.(\theta_{12,15,45}(s))=0
\]
and so
\[
\lambda_{45}(\theta_{12,15,45}(s))=1.
\]

Now we consider Equation \eqref{maxdeg3} with $a=1$, $b=3$, $c=5$, $q=2$, $p=4$, $H=(1,2)$ and $L=(4,5)$. We obtain
\[
\theta_{35,12,15}(s)+\theta_{51,12,35}(s)+2x_4\partial_3.(\theta_{51,12,45}(s))=0.
\]
Applying $x_3\partial_4$ to this equation we have
\[
-\theta_{45,12,15}(s)-\theta_{51,12,45}(s)+2h_{34}.(\theta_{51,12,45}(s))=0
\]
and from this we get $\lambda_{34}(\theta_{12,15,45}(s))=1$.

Finally, we use again Equation \eqref{maxdeg3} with $a=1$, $b=4$, $c=5$, $q=2$, $p=3$, $H=(1,2)$, $L=(1,5)$ which gives $2x_3\partial_1.(\theta_{45,12,15}(s))=0$, hence
\[
\lambda_{13}(\theta_{12,15,45}(s))=0
\]
proving that $\mu=(0,0,1,1)$. It follows that $\lambda=(1,1,0,0)$ since $\lambda(\theta_{12,15,45})=(-1,-1,1,1)$.

By Remark \ref{asdf} we have $-2\theta^3_{15}(s)-\theta_{12,15,45}(s)=0$ hence the leading term of the singular vector $\varphi(1\otimes s)$ is
$\omega_{12,15,45}\otimes \theta_{12,15,45}(s)+\partial_3d_{15}\otimes\theta^3_{15}(s)= d_{12}d_{15}d_{45}\otimes \theta_{12,15,45}(s)$.
It follows that $\varphi=\nabla_C\nabla_B\nabla_A$ due to Proposition \ref{leading}.
\end{proof} 
In the next result, for notational convenience, for all $a,b\in [5]$ we let $(-1)^{a<b}=(-1)^{\chi_{a<b}}$.
\begin{proposition}\label{12equazioni}
Let $\{x,y,z,w,t\}=[5]$ and let $s$ be a highest weight vector in $F(\lambda)$. Assume that $\theta_{xy,xz,wt}$ has the leading weight of $\varphi$. Then 
the following equations hold:
\begin{align}
\label{eq1} -2\varepsilon_{xyzwt}\theta^y_{xy}(s)+(-1)^{y<t}\theta_{xz,xt,yw}(s)&+(-1)^{y<t}\theta_{xz,xy,tw}(s)+ (-1)^{y<z}\theta_{xz,xt,yw}(s)\\
\nonumber & +(-1)^{y<z}\theta_{xy,xt,zw}(s)+(-1)^{y<x}\theta_{xy,xw,zt}(s)=0
\end{align}
\begin{align}
\label{eq2} 2\varepsilon_{xyzwt}\theta^y_{xy}(s)+(-1)^{y<t}\theta_{xw,xt,yz}(s)&+(-1)^{y<t}\theta_{xw,xy,tz}(s)+ (-1)^{y<w}\theta_{xw,xt,yz}(s)\\
\nonumber &+(-1)^{y<w}\theta_{xy,xt,wz}(s)+(-1)^{y<x}\theta_{xy,xz,wt}(s)=0
\end{align}
\begin{align}
\label{eq3} -2\varepsilon_{xyzwt}\theta^y_{xy}(s)+(-1)^{y<z}\theta_{xw,xz,yt}(s)&+(-1)^{y<z}\theta_{xw,xy,zt}(s)+ (-1)^{y<w}\theta_{xw,xz,yt}(s)\\
\nonumber &+(-1)^{y<w}\theta_{xy,xz,wt}(s)+(-1)^{y<x}\theta_{xy,xt,wz}(s)=0
\end{align}
\begin{align}
\label{eq4} 2\varepsilon_{xyzwt}\theta^z_{xz}(s)+(-1)^{z<t}\theta_{xy,xt,zw}(s)&+(-1)^{z<t}\theta_{xy,xz,tw}(s)+ (-1)^{z<y}\theta_{xy,xt,zw}(s)\\
\nonumber &+(-1)^{z<y}\theta_{xz,xt,yw}(s)+(-1)^{z<x}\theta_{xz,xw,yt}(s)=0
\end{align}
\begin{align}
\label{eq5} -2\varepsilon_{xyzwt}\theta^z_{xz}(s)+(-1)^{z<t}\theta_{xw,xt,zy}(s)&+(-1)^{z<t}\theta_{xw,xz,ty}(s)+ (-1)^{z<w}\theta_{xw,xt,zy}(s)\\
\nonumber &+(-1)^{z<w}\theta_{xz,xt,wy}(s)+(-1)^{z<x}\theta_{xz,xy,wt}(s)=0
\end{align}
\begin{align}
\label{eq6} 2\varepsilon_{xyzwt}\theta^z_{xz}(s)+(-1)^{z<y}\theta_{xw,xy,zt}(s)&+(-1)^{z<y}\theta_{xw,xz,yt}(s)+ (-1)^{z<w}\theta_{xw,xy,zt}(s)\\
\nonumber &+(-1)^{z<w}\theta_{xz,xy,wt}(s)+(-1)^{z<x}\theta_{xz,xt,wy}(s)=0
\end{align}
\begin{align}
\label{eq7} 2\varepsilon_{xyzwt}\theta^t_{xt}(s)+(-1)^{t<y}\theta_{xz,xy,tw}(s)&+(-1)^{t<y}\theta_{xz,xt,yw}(s)+ (-1)^{t<z}\theta_{xz,xy,tw}(s)\\
\nonumber &+(-1)^{t<z}\theta_{xt,xy,zw}(s)+(-1)^{t<x}\theta_{xt,xw,zy}(s)=0
\end{align}
\begin{align}
\label{eq8} -2\varepsilon_{xyzwt}\theta^t_{xt}(s)+(-1)^{t<w}\theta_{xz,xw,ty}(s)&+(-1)^{t<w}\theta_{xz,xt,wy}(s)+ (-1)^{t<z}\theta_{xz,xw,ty}(s)\\
\nonumber &+(-1)^{t<z}\theta_{xt,xw,zy}(s)+(-1)^{t<x}\theta_{xt,xy,zw}(s)=0
\end{align}
\begin{align}
\label{eq9} 2\varepsilon_{xyzwt}\theta^t_{xt}(s)+(-1)^{t<w}\theta_{xy,xw,tz}(s)&+(-1)^{t<w}\theta_{xy,xt,wz}(s)+ (-1)^{t<y}\theta_{xy,xw,tz}(s)\\
\nonumber &+(-1)^{t<y}\theta_{xt,xw,yz}(s)+(-1)^{t<x}\theta_{xt,xz,yw}(s)=0
\end{align}
\begin{align}
\label{eq10} -2\varepsilon_{xyzwt}\theta^w_{xw}(s)+(-1)^{w<t}\theta_{xy,xt,wz}(s)&+(-1)^{w<t}\theta_{xy,xw,tz}(s)+ (-1)^{w<y}\theta_{xy,xt,wz}(s)\\
\nonumber &+(-1)^{w<y}\theta_{xw,xt,yz}(s)+(-1)^{w<x}\theta_{xw,xz,yt}(s)=0
\end{align}
\begin{align}
\label{eq11} 2\varepsilon_{xyzwt}\theta^w_{xw}(s)+(-1)^{w<t}\theta_{xz,xt,wy}(s)&+(-1)^{w<t}\theta_{xz,xw,ty}(s)+ (-1)^{w<z}\theta_{xz,xt,wy}(s)\\
\nonumber &+(-1)^{w<z}\theta_{xw,xt,zy}(s)+(-1)^{w<x}\theta_{xw,xy,zt}(s)=0
\end{align}
\begin{align}
\label{eq12} -2\varepsilon_{xyzwt}\theta^w_{xw}(s)+(-1)^{w<y}\theta_{xz,xy,wt}(s)&+(-1)^{w<y}\theta_{xz,xw,yt}(s)+ (-1)^{w<z}\theta_{xz,xy,wt}(s)\\
\nonumber &+(-1)^{w<z}\theta_{xw,xy,zt}(s)+(-1)^{w<x}\theta_{xw,xt,zy}(s)=0
\end{align}
\end{proposition}
\begin{proof} We use Remark \ref{asdf} twelve times  with $L=Q=(p,q)$ any ordered pair in $\{y,z,w,t\}$ and $H=(x,p)$ to obtain the stated equations. 
More precisely  we get Equation \eqref{eq1} with $p=y$, $q=w$;
Equation \eqref{eq2} with $p=y$, $q=z$;
Equation \eqref{eq3} with $p=y$, $q=t$;
Equation \eqref{eq4} with $p=z$, $q=w$;
Equation \eqref{eq5} with $p=z$, $q=y$;
Equation \eqref{eq6} with $p=z$, $q=t$; 
Equation \eqref{eq7} with $p=t$, $q=w$;
Equation \eqref{eq8} with $p=t$, $q=y$;
Equation \eqref{eq9} with $p=t$, $q=z$;
Equation \eqref{eq10} with $p=w$, $q=z$;
Equation \eqref{eq11} with $p=w$, $q=y$;
Equation \eqref{eq12} with $p=w$, $q=t$.
\end{proof}

Proposition \ref{12equazioni} provides 12 linear equations in the ten unknown
$\theta_{xy,xz,wt}(s)=f_{wt}$, $\theta_{xy,xw,zt}(s)=f_{zt}$, $\theta_{xy,xt,zw}(s)=f_{zw}$, $\theta_{xz,xw,yt}(s)=f_{yt}$, $\theta_{xz,xt,yw}(s)=f_{yw}$, $\theta_{xw,xt,yz}(s)=f_{yz}$, 
$\varepsilon_{xyzwt}\theta^y_{xy}(s)=b_y$, $\varepsilon_{xyzwt}\theta^{z}_{xz}(s)=b_z$, $\varepsilon_{xyzwt}\theta^w_{xw}(s)=b_w$, $\varepsilon_{xyzwt}\theta^t_{xt}(s)=b_t$.
We are now interested in the study of the weights $\lambda_{i,j}(\theta_{xy,xz,wt}(s))$. 
 \begin{proposition}\label{aaa} Let $\{p,q,a,b,c\}=[5]$ with $c<p$, let $s$ be a highest weight vector in $F(\lambda)$, $H,L\in\mathcal{I}_1$ 
and assume that $\theta_{ab,H,L}$ has the leading weight of $\varphi$. Then we have 
 \begin{align*}
  &2h_{cp}.(\theta_{ab,H,L}(s))=\\
	&-2\varepsilon_{pqabc}\chi_{L\in B_1Q}(x_c\partial_p.\theta^p_{H})(s)+(x_c\partial_p.(x_p\partial_c \theta_{ab,H,L}))(s)+(-1)^{\chi_{p<b}}(x_c\partial_p.(x_p\partial_{b}.\theta_{ca,H,L}))(s)\\
  &+(-1)^{\chi_{p<a}}(x_c\partial_p.(x_p\partial_{a}.\theta_{bc,H,L}))(s)-2\chi_{c<b<p}(x_c\partial_b.\theta_{ca,H,L})(s)-2\chi_{c<a<p}(x_c\partial_a.\theta_{bc,H,L})(s)
 \end{align*}
\end{proposition}
\begin{proof}
Equation \eqref{maxdeg3} is equivalent to the following
 \begin{align*}
 2\varepsilon_{pqabc}\chi_{L\in B_1Q}\theta^p_{H}(s)&-(x_p\partial_{c}.\theta_{ab,H,L})(s)+(-1)^{\chi_{p>b}}(x_p\partial_{b}.\theta_{ca,H,L})(s)+(-1)^{\chi_{p>a}}(x_p\partial_{a}.\theta_{bc,H,L})(s)\\
 &+ 2x_p\partial_c.(\theta_{ab,H,L}(s))+2\chi_{p>b}x_p\partial_b.(\theta_{ca,H,L}(s))+2\chi_{p>a}x_p\partial_a.(\theta_{bc,H,L}(s))=0.
 \end{align*}
 We apply $x_c\partial_p$ to this equation and we obtain
 \begin{align*}
 2&\varepsilon_{pqabc}\chi_{L\in B_1Q}x_c\partial_p.(\theta^p_{H}(s))-(x_c\partial_p.(x_p\partial_c.\theta_{ab,H,L}))(s)+(-1)^{\chi_{p>b}}(x_c\partial_p.(x_p\partial_{b}.\theta_{ca,H,L}))(s)\\
 &+(-1)^{\chi_{p>a}}(x_c\partial_p.(x_p\partial_{a}.\theta_{bc,H,L}))(s)+ 2h_{cp}.(\theta_{ab,H,L}(s))+2\chi_{c<b<p}(x_c\partial_b.\theta_{ca,H,L})(s)\\
&+2\chi_{c<a<p}(x_c\partial_a.\theta_{bc,H,L})(s)=0.
 \end{align*}
The result follows.
\end{proof}
 
\begin{corollary} Let $\{x,y,z,w,t\}=[5]$ and assume that $\theta_{xy,xz,wt}$ has the leading weight of $\varphi$. Then we have

\noindent
 if $z<w$,
 \begin{align}
\label{9141}  2h_{zw}.f_{wt}=2(b_w-b_z)+f_{zt}+f_{wt}+(-1)^{\chi_{w<x}}(f_{yw}+f_{yz})-2\chi_{z<x<w}f_{wt};
 \end{align}
 if $y<z$,
 \begin{align}
\label{9142} 2h_{yz}.f_{wt}&=2(b_y-b_z)+(-1)^{\chi_{t<z}}(-f_{zw}-f_{yw})+(-1)^{\chi_{w<z}}(f_{zt}+f_{yt})\\ &
 \nonumber -2\chi_{y<t<z}(f_{wt}+f_{yw})-2\chi_{y<w<z}(f_{wt}-f_{yt})
 \end{align} 
 if $w<t$,
 \begin{align}
 \label{9143}2h_{wt}.f_{wt}&=(-1)^{\chi_{y<t}}(f_{yw}+f_{yt})+(-1)^{\chi_{x<t}}(f_{yt}+f_{yw})\\
\nonumber &-2\chi_{w<y<t}(f_{wt}-f_{yt})-2\chi_{w<x<t}(f_{wt}-f_{yw}).
 \end{align}
if $w<z$,
 \begin{align}
\label{9144} 2h_{wz}.f_{wt}&=f_{wt}+f_{zt}+(-1)^{\chi_{z<y}}(f_{wt}+f_{zt})+2\chi_{w<y<z}(-f_{wt}+f_{yt})+
2\chi_{w<x<z}(-f_{wt}+f_{yw})
 \end{align}
if $x<y$,
 \begin{align}
\label{9145} 2h_{xy}.f_{wt}&=((-1)^{\chi_{y<t}}+(-1)^{\chi_{y<w}})(-f_{yw}+f_{yt}+f_{yz})\\
\nonumber &-2\chi_{x<t<y}(f_{wt}+f_{yt}-f_{zt})
-2\chi_{x<w<y}(f_{wt}+f_{yw}-f_{yz})
 \end{align}
\end{corollary}
\begin{proof}
The statement follows from Proposition \ref{aaa} with the following choices:
 \begin{enumerate}
 \item  $a=x$ $b=y$, $c=z$, $p=w$, $q=t$, $H=(x,z)$, $L=(w,t)$. 
 \item  $c=y$, $p=z$, $a=w$, $b=t$, $q=x$, $H=(x,y)$, $L=(z,x)$.
 \item  $c=w$, $p=t$, $a=x$, $b=y$, $q=z$, $H=(x,z)$, $L=(w,t)$.
 \item  $c=w$, $p=z$, $a=x$, $b=y$, $q=t$, $H=(x,z)$, $L=(w,t)$.
 \item  $c=x$, $p=y$, $a=w$ $b=t$, $q=z$, $H=(x,y)$, $L=(x,z)$.
 \end{enumerate}
\end{proof}
 \begin{proposition}\label{trequarti}
Let $s$ be a highest weight vector in $F(\lambda)$. For $c<p$ we have
\begin{align*}
4h_{cp}.(\theta^a_{ab}(s))&=(-4\chi_{c<b<p}-4\chi_{c<a<p})\theta^a_{ab}(s)+(2-4\chi_{c<a})\theta^c_{bc}(s)+(-2+4\chi_{p<a})\theta^p_{bp}(s)\\
&\hspace{5mm}+\varepsilon_{pqabc}\big(\theta_{ab,bp,cq}(s)+\theta_{ab,bc,pq}(s)+\theta_{ap,cb,bq}(s)+\theta_{ac,pb,bq}(s)\big)
\end{align*}

 \end{proposition}
\begin{proof}
 We start from Equation \eqref{deg3n2}:
 \begin{equation}\label{coefparta}
  \theta_{ab,bc,cq}(s)+\theta_{ac,cb,bq}(s)+\varepsilon_{pqabc}\big(-2\theta^p_{bc}(s)+4x_p\partial_c.(\theta^a_{ab}(s))+4x_p\partial_b.(\theta^a_{ca}(s))+4x_p\partial_a. (\theta^a_{bc}(s))\big)=0.
 \end{equation}
 We want to apply $x_c\partial_p$ to this equation and so we do the following two  preliminary calculations:
\begin{align*}
x_c\partial_p.(x_p\partial_b.(\theta^a_{ca}(s)))&=\chi_{c<b}x_c\partial_p.(x_p\partial_b.(\theta^a_{ca}(s)))\\
&=\chi_{c<b}x_c\partial_b.(\theta^a_{ca}(s))+\chi_{c<b}x_p\partial_b.(x_c\partial_p.(\theta^a_{ca}(s)))\\
&=-\chi_{c<b}\theta^a_{ba}(s)-\chi_{c<b}x_p\partial_b.(\theta^a_{pa}(s))\\
&= \chi_{c<b}\theta^a_{ab}(s)+\chi_{c<b}\chi_{p<b}\theta^a_{ba}(s)\\
&=\chi_{c<b}(1-\chi_{p<b})\theta^a_{ab}(s)\\
&=\chi_{c<b<p}\theta^a_{ab}(s)
\end{align*}

\begin{align*}
x_c\partial_p.(x_p\partial_a.(\theta^a_{bc}(s)))&=\chi_{c<a}x_c\partial_p.(x_p\partial_a.(\theta^a_{bc}(s)))\\
&=\chi_{c<a}x_c\partial_a.(\theta^a_{bc}(s))+\chi_{c<a}x_p\partial_a.(x_c\partial_p.(\theta^a_{bc}(s)))\\
&=\chi_{c<a}(\theta^c_{bc}(s)-\theta^a_{ba}(s))-\chi_{c<a}x_p\partial_a.(\theta^a_{bp}(s))\\
&=\chi_{c<a}\theta^c_{bc}(s)+\chi_{c<a}\theta^a_{ab}(s)-\chi_{c<a}\chi_{p<a}(\theta^p_{bp}(s)-\theta^a_{ba}(s))\\
&=\chi_{c<a}\theta^c_{bc}(s)-\chi_{p<a}\theta^p_{bp}(s)+\chi_{c<a<p}\theta^a_{ab}(s)
\end{align*}

Therefore, if we apply $x_c\partial_p$ to Equation \eqref{coefparta}, using the previous computations, we obtain
 \begin{align*}
  -&\theta_{ab,bp,cq}(s)-\theta_{ab,bc,pq}(s)-\theta_{ap,cb,bq}(s)-\theta_{ac,pb,bq}(s)+\varepsilon_{pqabc}\big(-2\theta^c_{bc}(s)+2\theta^p_{bp}(s)\\&+4h_{cp}.(\theta^a_{ab}(s))+4\chi_{c<b<p}\theta^a_{ab}(s)+4\chi_{c<a}\theta^c_{bc}(s)-4\chi_{p<a}\theta^p_{bp}(s)+4\chi_{c<a<p}\theta^a_{ab}(s)\big)=0
 \end{align*}
hence we get the statement.
\end{proof}

\begin{proposition} \label{xyxzwt}Let $\{h,k,l,m,n\}=[5]$. Then  $\theta_{hk,hl,mn}$ and $\theta^k_{hk}$ do not have the leading weight of $\varphi$.
\end{proposition} 
\begin{proof}
We first assume $h=1$ and we let $x=1$, $y=2$, $z=3$, $w=4$, $t=5$. We use notation introduced after the proof of Proposition \ref{12equazioni}
and we observe that, up to a sign, $\theta_{1k,1l,mn}(s)\in \{f_{23}, f_{24}, f_{25}, f_{34}, f_{35}, f_{45}\}$ and $\theta^k_{1k}(s)\in\{b_2,b_3,b_4,b_5\}$.
We solve the linear system provided by Proposition \ref{12equazioni} and
we have:
\begin{itemize}
\item $f_{35}=-f_{45}=-f_{34}$
\item $f_{24}=-f_{25}=-f_{23}$
\item $2b_2=-3f_{34}+2f_{23}$
\item $2b_3=2b_5=2b_4=-f_{34}$.
\end{itemize}

We use Proposition \ref{trequarti} with $a=4$, $b=1$, $c=2$, $p=3$, $q=5$ and we obtain
\begin{align*}
  h_{23}.b_4= &\frac{1}{2}b_2-\frac{1}{2}b_3+\frac{1}{4}(f_{25}+f_{35}+f_{34}+f_{24})\\
  &=\frac{1}{4}(-3f_{34}+2f_{23})+\frac{1}{4}f_{34}+\frac{1}{4}(f_{23}-f_{34}+f_{34}-f_{23})\\
  &=-\frac{1}{2}f_{34}+\frac{1}{2}f_{23}
 \end{align*}
 therefore
 \[
  h_{23}.f_{34}=f_{34}-f_{23}.
 \]
 Now we use Equation \eqref{9142}:
 \[
  2h_{23}.f_{45}=2(b_2-b_3)-f_{34}-f_{24}+(f_{35}+f_{25})
 \]
 i.e.
 \[
  2h_{23}.f_{34}=-3f_{34}+2f_{23}+f_{34}-f_{34}+f_{23}-f_{34}+f_{23}=-4f_{34}+4f_{23}
 \]
 or
 \[
  h_{23}.f_{34}=-2f_{34}+2f_{23}.
 \]
Comparing this with the previous equation we obtain $f_{34}=f_{23}$.

Now we use Equation \eqref{9145}:
\[
2h_{12}.f_{45}=2f_{24}-2f_{25}-2f_{23}
\]
i.e.
\[
2h_{12}.f_{34}=-2f_{23}-2f_{23}-2f_{23}=-6f_{34}
\]
This implies that $f_{34}=f_{23}=0$. It follows that $\theta_{1k,1l,mn}(s)=0$ and $\theta^k_{1k}(s)=0$.

\bigskip

Now let $h=2$ and  $x=2$, $y=1$, $z=3$, $w=4$, $t=5$. Similarly as above
we have, up to a sign, $\theta_{2k,2l,mn}(s)\in \{f_{13}, f_{14}, f_{15}, f_{34}, f_{35}, f_{45}\}$ and $\theta^k_{2k}(s)\in\{b_1,b_3,b_4,b_5\}$.
We solve the linear system provided by Proposition \ref{12equazioni} and
we have:
\begin{itemize}
\item $f_{35}=-f_{45}=-f_{34}$
\item $f_{14}=-f_{15}=-f_{13}$
\item $2b_1=-f_{34}+2f_{13}$
\item $2b_3=2b_4=2b_5=-f_{34}$
\end{itemize}

We use Proposition \ref{trequarti} with $a=4$, $b=2$, $c=1$, $p=5$, $q=3$ and we obtain:
\[
h_{15}.b_4=\frac{1}{2}f_{34}+\frac{1}{2}f_{13}
\]
i.e.,
\[
h_{15}.f_{34}=-f_{34}-f_{13}
\]

Now we use Equations \eqref{9141}, \eqref{9142}, \eqref{9143} and we obtain:
\[
h_{15}.f_{34}=2f_{13}-f_{34}\]

It follows that:
\[
2f_{13}-f_{34}=-f_{34}-f_{13}\]
i.e., $f_{13}=0$, hence $h_{15}.f_{34}=-f_{34}$ which implies $f_{34}=0$. It follows that $\theta_{2k,2l,mn}(s)=0$ and $\theta^k_{2k}=0$.

\bigskip

Now let $h=3$ and  $x=3$, $y=1$, $z=2$, $w=4$, $t=5$. Similarly as above
we have, up to a sign, $\theta_{3k,3l,mn}(s)\in \{f_{12}, f_{14}, f_{15}, f_{24}, f_{25}, f_{45}\}$ and $\theta^k_{3k}(s)\in\{b_1,b_2,b_4,b_5\}$.
We solve the linear system provided by Proposition \ref{12equazioni} and
we have:
\begin{itemize}
\item $f_{15}=f_{24}=-f_{25}=-f_{14}$
\item $f_{45}=-2b_4=-2b_5=-2f_{14}-f_{12}$
\item $2b_1=2b_2=f_{12}$
\end{itemize}

We use Proposition \ref{trequarti} with $a=2$, $b=3$, $c=1$, $p=5$, $q=4$ and we obtain:
\[
-h_{15}(b_2)=\frac{1}{2}f_{12}-\frac{1}{2}f_{14}
\]
i.e.,
\[
h_{15}(f_{12})=f_{14}-f_{12}
\]

Now we use Equations \eqref{9141}, \eqref{9142}, \eqref{9143} and we obtain:
\[h_{15}(f_{45})=3f_{14}+f_{12}.\]

It follows that:
\[h_{15}.f_{14}=-\frac{1}{2}h_{15}.(f_{45}+f_{12})=-2f_{14}\]
hence $f_{14}=0$ and $h_{15}.f_{12}=-f_{12}$ from which it follows that $f_{12}=0$.
We conclude that $\theta_{3k,3l,mn}(s)=0$ and $\theta^k_{3k}(s)=0$.

If $h=4,5$ the result follows from Corollary \ref{degree3dual}.
\end{proof}
Now we can summarize the classification of morphisms of degree 3 in the next result.
\begin{theorem}\label{teorema3}
 Let $\varphi:M(\lambda)\rightarrow M(\mu)$ be a morphism of degree 3. Then $\lambda=(1,1,0,0)$, $\mu=(0,0,1,1)$ and up to a scalar $\varphi=\nabla_C\nabla_B\nabla_C$.
\end{theorem}
\begin{proof}
 This follows from Lemmas \ref{a,b,c,d}, \ref{a,b,c}, \ref{12e45}, \ref{1245}, \ref{morfismogrado3} and Proposition \ref{xyxzwt}.
\end{proof}

\end{document}